\documentclass[leqno,12pt]{amsart}
\usepackage{amsmath,amsthm, mathrsfs, amssymb,amsfonts}
\usepackage{color}
\usepackage[all]{xy}
\usepackage{enumitem}

\usepackage[left=1 in,top=1 in,right=1 in,bottom=1 in]{geometry}
\usepackage{hyperref}

\numberwithin{equation}{section}

\usepackage{mathtools}
\mathtoolsset{showonlyrefs}

\title[Navier and Stokes meet Poincar\' e and Dulac]{Navier and Stokes meet Poincar\' e and Dulac}

\author[C. Foias]{Ciprian Foias}
\address{Department of Mathematics,
Mailstop 3368,
Texas A\&M University,
College Station, TX 77843-3368, U. S. A.} 

\author[L.  Hoang]{Luan Hoang}
\address{Department of Mathematics and Statistics, Texas Tech University, Box 41042, Lubbock, TX 79409-1042, U. S. A.} 
\email{luan.hoang@ttu.edu}

\author[J.-C. Saut]{Jean-Claude Saut}
\address{Laboratoire de Math\' ematiques, UMR 8628,\\
Universit\' e Paris-Sud et CNRS, 91405 Orsay, France}
\email{jean-claude.saut@u-psud.fr}

\date\today
\begin{document}
	\newcommand{\R}{\mathbb{R}}
	\newcommand{\T}{\mathbb{T}}
	\newcommand{\Z}{\mathbb{Z}}
	\newcommand{\C}{\mathcal{C}}
	\newcommand{\U}{\mathbb{U}}
	\newcommand{\N}{\mathbb{N}}
	\newcommand{\W}{\mathcal{W}}
	\newcommand{\D}{\mathcal{D}}
	\newcommand{\s}{\text{sinh}}
	\newcommand{\norm}[1]{ \left \lVert #1  \right \rVert}
	\newcommand{\norma}[1]{ \left \lVert  #1 \right \rVert}
	\newcommand{\bra}[1]{\left \langle #1 \right \rangle}

\newcommand{\mysnote}[1]{\fbox{#1}}
\newcommand{\mybox}[1]{ \fbox{\parbox{.9\linewidth}{#1}} }
\newcommand{\highlight}[1]{{\color{black} #1}}
\newcommand{\hilite}[1]{{\color{blue} #1}}
\newcommand{\alert}[1]{{\color{red} #1}}

\def\mP{\mathcal P}
\def\mQ{\mathcal Q}
\def\mB{\mathcal B}
\def\mD{\mathcal D}
\def\mA{\mathcal A}
\def\eqdef{\stackrel{\rm def}{=}}
\newcommand{\sinorm}[2]{\big [\, #1\, \big ]^{#2}}
\newcommand{\dinorm}[2]{\left [ \left [\,  #1 \, \right]\right]_{#2}}
\def\vecvor{\omega}

\newcommand{\bvec}[1]{\mathbf{#1}}
\def\vecx{\bvec x}
\def\vecy{\bvec y}
\def\veca{\bvec a}
\def\vecb{\bvec b}
\def\vecc{\bvec c}
\def\vecu{\bvec u}
\def\vecv{\bvec v}
\def\vecw{\bvec w}
\def\veck{\bvec k}
\def\vecj{\bvec j}
\def\vecn{\bvec n}
\def\vece{\bvec e}
\def\vecm{\bvec m}

\def\heli{{\mathcal H}}
\def\helirate{{\mathcal J}}
\def\energy{{\mathcal E}}
\def\dissrate{{\mathcal F}}

\def\curlop{\mathfrak C}

\newcommand{\inprod}[2]{\langle{#1},{#2}\rangle}
\newcommand{\doubleinprod}[2]{\langle\!\langle{#1},{#2}\rangle\!\rangle}

\newcommand{\beq}{\begin{equation}}
\newcommand{\eeq}{\end{equation}}

\newcommand{\beqs}{\begin{equation*}}
\newcommand{\eeqs}{\end{equation*}}

\def\ddt{\frac{d}{dt}}

\newcommand{\bds}{\begin{displaystyle}}
\newcommand{\eds}{\end{displaystyle}}

\renewcommand{\Re}{\mathbf{Re}}
\renewcommand{\Im}{\mathbf{Im}}

\newcommand{\DA}{\mathcal{D}(A)}
\newcommand{\solnset}{\mathcal{R}}
\def\SAstar{\mathcal S_A^\star}

\def\myclearpage{}
\def\myclearpage{\clearpage}
\def\vol{L^3}


\def\NSE{Navier--Stokes equations}
\newcommand{\setdef}[1]{\left\{\,#1\,\right\}}

\newcommand{\Ghighlight}[1]{{\color{darkgreen} #1}}

\newtheorem{theorem}{Theorem}[section]
\newtheorem{lemma}{Lemma}[section]
\newtheorem{proposition}{Proposition}[section]
\newtheorem{corollary}[theorem]{Corollary}
\theoremstyle{remark}
\newtheorem{remark}{Remark}[section]
\theoremstyle{definition}
\newtheorem{definition}[theorem]{Definition}
\newtheorem{example}[theorem]{Example}

\def\kstar{{\veck^0}}
\def\manifold{\mathcal{M}}
\def\Hodd{H_{odd}}
\def\orthplane{\veca^\perp}
\def\directmanifold{\manifold_{\veca^\perp }}

\begin{abstract}
This paper surveys various precise (long-time) asymptotic results for the solutions of the Navier-Stokes equations with potential forces in bounded domains. It turns out that the asymptotic expansion leads surprisingly to a Poincar\' e-Dulac normal form of the Navier-Stokes equations.
We will also discuss some related results and a few open issues.
\end{abstract}

       \dedicatory{Dedicated to Claude-Michel Brauner on the occasion of his 70th birthday.}
       
\maketitle
\tableofcontents

\section{Introduction and some historical tidbits}
\label{secintro}

 Claude-Louis  Navier (1785--1836) and George B. Stokes  (1819--1903), of course, never met Henri Poincar\' e (1854--1912) and Henri Dulac (1870--1955) as we will recall in the small historical section below. (For more details, and more generally for a fascinating account  of the history of fluid dynamics, see \cite{Dar}). However their mathematical theory of dynamical systems and physics theory of fluid mechanics have finally met more than a hundred years after their initial contributions.

We will not comment on Stokes and Poincar\' e who are well-known scientists but make a few (may be not so well known) remarks on Dulac, and mainly on Navier.

 Claude Louis Marie Henri Navier was an "X-Ponts" engineer in the jargon of {\it Grandes Ecoles}, first trained at the Ecole Polytechnique, then at the Ecole des Ponts et Chauss\' ees, one of the {\it Ecoles d'applications} such as the {\it Ecole des Mines} (Augustin Louis Cauchy was an "X-Ponts", Henri  Poincar\' e was an "X-Mines"). He was in the main stream of French theoretical continuum mechanics of this time.

 A major figure of French Mechanics of this time, Adh\' emar Barr\' e de Saint-Venant (1797--1886), also an X-Ponts, was a former student and  successor of Navier. Among many other things, he derived the so-called {\it Saint-Venant} system (or {\it shallow water system}).
He was the advisor and protector of Joseph Boussinesq (1842--1929) who made fundamental contributions in Fluid Mechanics, in particular on the theory of water waves.

As noticed by Olivier Darrigol in his book \cite{Dar}, 

\begin{quote}
"Navier and other Polytechnicians' efforts to reconcile theoretical and applied mechanics had no clear effect on French engineering practice. Industry prospered much faster in Britain, despite the lesser mathematical training of its engineers. Some of Navier's colleagues saw this and ridiculed the use of transcendental mathematics in concrete problems of construction. In the mid-1820s, a spectacular incident apparently justified their disdain. Navier's chef-d'oeuvre, a magnificent suspended bridge at the Invalides, had to be dismantled in the final stage of its construction".
\end{quote}

 Actually Navier was probably most famous in his time for  the "disaster" of the {\it pont des Invalides}, the first suspended bridge over the Seine river. In fact Navier had mis-estimated  the direction of the  force exerted by the chain on the stone. This could have been corrected easily but the hostile municipal authorities decided the dismantlement of Navier bridge.  Honor\' e de Balzac is alluding to this incident (rather ironically)  in his novel  {\it Le cur\' e de village}:

\begin{quote}
"La France enti\`ere a vu le d\' esastre, au coeur de Paris, du premier pont suspendu que voulut \' elever un ing\' enieur, membre de l'Acad\' emie des Sciences,  triste chute qui fut caus\' ee par des fautes que ni le constructeur du canal de Briare, sous Henri IV, ni le moine qui a b\^{a}ti le Pont-Royal, n'eussent faites, et que l'Administration consola en l'appelant au Conseil G\' en\' eral (des Ponts et Chauss\' ees). Les Ecoles Sp\' eciales seraient-elles donc des fabriques d'incapacit\' es? Ce sujet exige de longues observations".
\end{quote}

\begin{quote}
\textit{(Translation)} "All  France knew of the disaster which happened in the heart of Paris to the first suspended bridge built by an engineer, a member of the Academy of Sciences, a melancholy collapse cause by blunders such as none of the ancient engineers, the man who cut the canal at Briare in Henry's IV time, or the monk who built the Pont Royal-would have made; but our administration consoled its engineer  for his blunder by making him a member of the  Council general. (of the {\it Ponts et Chauss\' ees}). Are our {\it Ecoles Sp\' eciales} producers of incapacities? This topic deserves lengthy observations".
\end{quote}

According to Saint-Venant however, the dismantlement of the bridge was more than  a local administrative deficiency: 

\begin{quote}
"At that time there already was  a surge of the spirit of denigration, not only of the  "savants" but also of science, disparaged under the name of {\it theory} opposed to {\it practice}; one henceforth exalted practice in its most material aspects, and prevented that higher mathematics could not help, as if, when it comes to results, it made sense to distinguish between the more or less elementary or transcendent procedures that led to them in an equally logical manner. Some "savants" supported or echoed these unfounded criticisms".
\end{quote}


As an aside, we note that there are some estimates of velocity fields in infinite dimensional spaces which have useful consequences, e.g. \cite{FMoTi}.

Navier was nevertheless a great scientist and, coming back to our subject,  he derived what are known as the Navier-Stokes equations in a 1823 M\' emoire. Further (different) derivations are due to Poisson (1831), Saint-Venant (1834) and Stokes (1843).

The last member of the quartet in title is the least famous of them. Henri Dulac (1870-1955), a former student of Ecole Polytechnique, was a professor at the University of Lyon and a corresponding member of the French Academy of Sciences. He was a specialist of the geometric theory of ordinary differential equations and developed in particular, after Poincar\' e, the theory of {\it normal forms}.

\subsection{The Navier-Stokes equations for viscous, incompressible fluid flows}
We now recall briefly  the derivation of  Navier-Stokes equations (NSE), based on conservation laws and the choice of a constitutive equation.
For complete background on NSE see e.g. \cite{LadyNSEbook,TemamNSEbook,TemamSIAMbook, CFbook,FMRTbook}.

We study fluid flows in Euclidean space of dimension $n=2,3$.
Let $\rho$ denote the density of the fluid, and $u$ its velocity.

\medskip
\indent $\bullet$ Conservation of mass:
$$\partial _t \rho +\text{div}(\rho u)=0.$$

We will consider only  the case when the density $\rho$ is constant, so that the conservation of mass reduces to 
$$\text{div}\;u=0.$$
We refer to this as the incompressibility condition.

\medskip
\indent $\bullet$ Conservation of momentum (Newton's law) for a general fluid:
$$\rho( \partial_t u+(u\cdot\nabla)u)=\text{div}(-\tilde p{\bf I}+{\bf T})+ \tilde f,$$ 
where $\tilde p$ is the (scalar) pressure, ${\bf T}$ is the {\it extra-stress} tensor, and $\tilde f$ represents body forces.
Here, we use the standard notation
$u\cdot\nabla=\sum_iu_i\partial_{x_i}.$

When ${\bf T}\equiv 0, $ one obtains the Euler equations (1755).

\medskip
\indent $\bullet$  Constitutive law: For a {\it Newtonian viscous fluid}, ${\bf T}$ at the present time $t$ is just proportional to the rate of deformation tensor 
${\bf D}(u)=(\nabla u+(\nabla u)^T)/2$ at time $t,$ that is 
$${\bf T}=\mu {\bf D}(u),$$
where $\mu$ is the dynamic viscosity coefficient. (For a general {\it non-Newtonian fluid}, ${\bf T}$ can be a complicated function of the past history of the deformations).

\medskip
Finally, one obtains the Navier-Stokes equations (NSE)
\begin{equation}\label{NSt}
    \left\lbrace
    \begin{array}{l}
 \partial_t u+(u\cdot\nabla)u-\nu \Delta u+\nabla p= f, \vspace{1mm}\\
      \text{div}\; u=0,
    \end{array}\right.
\end{equation}
where $\nu=\mu/\rho$ is the kinematic viscosity, $p=\tilde p/\rho$, and $f=\tilde f/\rho$. For simplicity, we will just call $\nu$ viscosity, $p$ pressure, and $f$ body force.

The system \eqref{NSt} consists of $(n+1)$ equations for $(n+1)$ unknowns, namely,  $u\in \R^n$ and $p\in \R$.
It will be completed with initial and boundary conditions in our considerations.

\subsection{Functional setting}
We consider the following two cases of fluid flows.

The first scenario is when the fluid is confined a smooth, bounded domain $\Omega$ of $\R^n$, and the velocity satisfies the no-slip boundary condition, i.e.,  $u=0$ on $\partial \Omega$.
We set in this case
$$\mathcal V=\lbrace v\in C^{\infty}_0(\Omega)^n:\;\text{div}\; v=0\rbrace.$$

The second scenario is when $(u,p)$ are defined in the whole space $\R^n$, but are $L$-periodic, for some $L>0$, in all there Cartesian coordinates.
Then $u$ and $p$ are considered as functions on the domain 
\beq\label{peridom} \Omega=\R^n/[0,L]^n.
\eeq 

We usually refer to this $\Omega$ as a periodic domain, and say $u$ and $p$ satisfy the periodicity boundary condition on $[0,L]^n$.
By a remarkable Galilean transformation, we assume, without 
loss of generality, that $u$  has zero averages over $\Omega$, i.e.,
\beqs
\int_\Omega u(x,t)dx=0.
\eeqs

We then define the space
$$\mathcal V=\Big\lbrace  \text{$\R^n$-valued $L$-periodic trigonometric polynomial $v$}:
\;\text{div}\;v=0,\;\int_{\Omega}v\, dx=0\Big\rbrace.$$

In both cases, we will use the classical spaces:
$$H=\text{closure of}\; \mathcal V \;\text{in}\; L^2(\Omega)^n,$$
$$V=\text{closure of}\; \mathcal V \;\text{in}\; H^1(\Omega)^n,$$
with norms 
$$\|u\|_H=|u|=\Big(\int_{\Omega}|u(x)|^2dx\Big)^{1/2},\quad \|u\|_V=\|u\|=\Big(\int_{\Omega}|\nabla v(x)|^2dx\Big)^{1/2}.$$

Note that notation $|\cdot|$ is used to denote the $H$-norm and the standard Euclidean norm on $\mathbb C^n$. However, its meaning will be clear in the context.

We denote the standard inner products of $L^2(\Omega)^k$, for $k\in\N$, by the same notation $\inprod{\cdot}{\cdot}$.

The norm in the Sobolev space $H^m(\Omega)$ is denoted by $\|\cdot\|_m.$ We also denote 
$$\mathcal E^m(\Omega)=H\cap H^m(\Omega)\text{ for $m\ge 0$, and }
\mathcal E^\infty(\Omega)=\bigcap_{m=0}^\infty \mathcal E^m(\Omega).$$

One has the Helmholtz-Leray decomposition for the case of no-slip boundary condition,
\begin{equation}\label{decomp}
L^2(\Omega)^n=H\oplus \{\nabla \varphi: \varphi\in H^1(\Omega) \}.
\end{equation}
and for the case of periodicity boundary condition,
\begin{equation}\label{peridecomp}
\Big \{ v\in L^2(\Omega)^n:\int_{\Omega}v\, dx=0\Big \}=H\oplus \{\nabla \varphi: \varphi\in H^1(\Omega) \}.
\end{equation}

We define $P$ to be the (Leray) orthogonal projection in $L^2(\Omega)^n$ onto $H$.

We assume at the moment that $(u,p)$ are classical solutions of NSE.
Thanks to \eqref{decomp} and \eqref{peridecomp}, we have $P(\nabla p)=0$. With this observation, we can reduce the unknowns of NSE from $(u,p)$ to $u$ only, by projecting the NSE to the space $H$. Having that in mind, we define the Stokes operator $A$ by 
$$Au=-P\Delta u$$
(with the ad hoc boundary conditions), and also define the bilinear form  
$$B(v,w)=P\lbrack (v\cdot \nabla)w\rbrack.$$

Assume  $f$ is a potential, i.e.,  $f=-\nabla \psi$, then, thanks to \eqref{decomp} and \eqref{peridecomp} again, $Pf=0$. 

Hence, applying the Leray projection $P$ to the NSE, and using the decomposition \eqref{decomp} or \eqref{peridecomp}, we rewrite the NSE \eqref{NSt} in the \textit{functional form} as:
\begin{equation}\label{NS}
\left\lbrace\begin{aligned}
   &\frac{du}{dt}+\nu Au+B(u,u)=0,\\
   &u(0)=u_0,
 \end{aligned}\right.
\end{equation}
where $u_0$ is a given initial data in $H$.

This functional form \eqref{NS} will be the focus of our study in this paper. 

\subsection{Basic facts}\label{basic}
The Stokes operator $A$ is an unbounded, self-adjoint operator in $H$ with domain 
$$D(A)=V\cap H^2(\Omega)^n.$$ 

Its spectrum $\sigma (A)$ consists of an unbounded sequence of real eigenvalues
\beq\label{Stkspec}
0<\Lambda_1<\Lambda_2<\ldots<\Lambda_k<\ldots,
\eeq
with corresponding multiplicities $m_1,m_2,\ldots,m_k,\ldots$ (See e.g. \cite{Catta}.)

The orthogonal projection in $H$ on the eigenspace of $A$ corresponding to $\Lambda_j$ will be denoted by $R_j$. 
 
\medskip
We  denote 
\beqs \mathscr S(A)=\{0<\mu_1=\Lambda_1<\mu_2<\mu_3<\ldots\},
\eeqs 
the additive semi-group generated by the $\Lambda_k$'s.

\medskip
In the periodic case, 
\beqs
\sigma(A)=\{ 4\pi^2|\veck|^2/L^2: \; \veck\in\Z^n,\; \veck\ne 0\},
\eeqs
hence,
\beq\label{periLamb}
\Lambda_1=4\pi^2/L^2\quad\text{and} \quad \sigma(A)\subset\{n\Lambda_1:\; n\in\N\}.
\eeq

By scaling the spatial and time variables, we can further assume, without loss of generality, 
\beq\label{scalength} \nu =1 \quad \text{and}\quad L=2\pi,
\eeq 
thus, 
\beq\label{scalesig} 
\Lambda_1=1,\quad \sigma(A)\subset \N,\quad \text{and}\quad \mathscr S(A)=\N.
\eeq

\bigskip
It has been known  since Leray's fundamental  papers (\cite{Ler1, Ler2, Ler3}) that 
\begin{enumerate}[label={\rm (\roman*)}]
\item For every initial data $u_0\in H, $ problem \eqref{NS} has a (Leray-Hopf) weak solution $u$ (see e.g. \cite{LionsBook,CFbook, LadyNSEbook,TemamNSEbook,TemamSIAMbook, FMRTbook}), that is,
\beqs
u\in C([0,\infty);H_{\rm w})\cap L^2_{\rm loc}([0,\infty);V),\quad u'\in L^{4/3}_{\rm loc}([0,\infty);V'),
\eeqs
satisfying \eqref{NS} in the dual space $V'$ of $V$,  
and the energy inequality
\beqs
\frac12|u(t)|^2+\int_{t_0}^t \|u(\tau)\|^2d\tau\le \frac12|u(t_0)|^2
\eeqs
holds for $t_0=0$ and almost all $t_0\in(0,\infty)$, and all $t\ge t_0$. 

Above, $H_{{\rm w}}$ denotes the space $H$ endowed with the weak topology. 

If $I$ is a closed interval in $\lbrack 0,\infty)$, then a weak solution $u$ is {\it regular} on $I$ if $u\in C(I;V).$

\item This weak solution becomes regular on $[T_0,\infty)$, for some $T_0=T_0(\nu, u_0)\ge 0$.
\item It is not known whether a (Leray-Hopf) weak solution is unique. 
\item If $u$ is regular on $I=\lbrack t_0,t_1 \rbrack$, then $u$ is uniquely determined on $I$ by $u(t_0).$
\item It is known that any (Leray-Hopf) weak solution $u$ satisfies
$$\frac{1}{2}\frac{d}{dt}|u(t)|^2+\nu\|u(t)\|^2\leq 0 \text{ in the distribution sense on }(0,\infty).$$ 
\item Any regular solution $u$ on $[0,\infty)$ satisfies the equation
\begin{equation}\label{ee}
\frac{1}{2}\frac{d}{dt}|u(t)|^2+\nu\|u(t)\|^2=0\text{ on }(0,\infty). 
\end{equation}
\end{enumerate}

\medskip
Because of property (ii) above, and that our goal is to study long-time behavior of solutions to NSE, we will, without loss of generality, mainly consider 
regular solutions on $[0,\infty)$. Let $\mathcal R$ denote the set of initial data in $V$ leading to global regular solutions.
Then $\mathcal R$ is an open subset of $V$, and, particularly,  $\mathcal R=V$ when $d=2$.

Obviously, $u=0$ is a trivial regular solution on $[0,\infty)$. Hence, $\mathcal R$ contains a neighborhood of $0$.
However, proving or disproving that $\mathcal R=V$ when $d=3$ is still an outstanding open problem.

Here afterward, we will call a regular solution $u$ on $[0,\infty)$, that is when $u(0)\in \mathcal R$, simply a {\it regular solution}.

\medskip
For a regular solution $u$, one has from \eqref{ee} and the Poincar\'e inequality, i.e., $$\Lambda_1|u|^2\leq\|u\|^2,$$  that
$$|u(t)|^2\leq |u_0|^2e^{-2\nu \Lambda_1 t},\quad\forall t\ge 0.$$

That is $|u(t)|^2$ must decay exponentially as $t\to\infty$ at the rate at least  $2\nu\Lambda_1$.

\subsection{Aim and outline of the paper}
A natural question (raised by P. Lax to C. Foias) is then to ask whether or not this decay rate is optimal.
In an early work, Dyer and Edmunds \cite{DE1968} prove that any non-trivial, regular solution $u$ has $|u(t)|^2$ also bounded below by an exponential function of $t$.
However, this answer is far from being definitive in describing the exact asymptotic behavior of a non-trivial, regular solution.  
In the following sections, we present the mathematical developments of the problem which lead to the asymptotic expansion and normal form theory (for NSE).  

\medskip
The paper is organized as follows. 
In section \ref{secDirichlet}, the Dirichlet quotient is proved to converge, as $t\to\infty$ to an eigenvalue of the Stokes operator. The asymptotic expansion of the regular solutions are studied. The set $\mathcal R$ is decomposed into nonlinear manifolds $M_k$'s, which characterize the rate of the decay for the solutions.
In section \ref{secexpansion}, each regular solution is proved to admit an asymptotic expansion in terms of exponential decays and polynomials in time. The application to analysis of the helicity is also presented.
In section \ref{secODE}, we review the classical Poincar\'e-Dulac theory of normal forms for ODEs.
In section \ref{secnormal}, it is shown that the asymptotic expansion reduces to a normal form, which, originally, is in a Fr\'echet space with very weak topology. It is then studied in suitable Banach spaces. In such a weighted normed space, the normalization map is continuous and the normal form for NSE is a well-posed infinite dimensional ordinary differential equation (ODE) system.
In section \ref{secPD}, the inverse of the normalization map is written as a formal power series in $E^\infty$, an appropriate topological vector subspace of $C^\infty$. It is then used to reduce the NSE to a Poincar\'e-Dulac normal form on $E^\infty$.
In section \ref{secFinal}, we review more related results and pose some open questions.


\section{Limit of the  Dirichlet quotients }
\label{secDirichlet}

By re-writing the energy equality \eqref{ee} in the form 
$$\frac{1}{2}\frac{d}{dt}|u|^2+\nu\frac{\|u\|^2}{|u|^2} |u|^2=0,$$
it is natural to study the limit as $t\to \infty$ of the  {\it  Dirichlet quotients} $$\lambda(t)= \frac{\|u(t)\|^2}{|u(t)|^2}.$$

\medskip
This is the beginning of a long process leading eventually to a  {\it  normal form } of  NSE.
One has the following results  (\cite{FS1, FS2}).

\begin{theorem}[\cite{FS1, FS2}]\label{Dirichlet}
	Let $u_0\in \mathcal R\setminus\{0\}$. 
	\begin{enumerate}[label={\rm (\roman*)}]
	\item $\lim_{t\to \infty}\lambda (t)=\Lambda(u_0)$ exists and belongs to $\sigma(A)$.
	\item $\lim_{t\to \infty}e^{\nu \Lambda(u_0)t} u(t)$ exists and belongs to $R_{\Lambda(u_0)}H$.
	\item There exist analytic submanifolds $M_k,$ $k=1,2,\ldots,$ of $\mathcal R$ having codimension $m_1+m_2+\ldots+m_k$ such that
	 $$\mathcal R=M_0\supset M_1\supset M_2\supset\ldots$$
	\item $M_k$ is invariant by the nonlinear semi-group $S(t)$ generated by the Navier-Stokes equation, that is $S(t)M_k\subset M_k,$ $\forall t\geq 0.$
	\item $u_0\in M_{k-1}\setminus M_k$ if and only if $\Lambda(u_0)=\Lambda_k,$ for $k=1,2,\ldots$. Consequently, $u_0\in M_{k-1}$ if and only if $\Lambda(u_0)\ge \Lambda_k$.
\item The tangent space of $M_k$ at $0$ is $M_k^{\rm lin}$, for $k=1,2,\ldots,$ where 
$$M_k^{\rm lin}=\lbrace u_0\in V; R_{\Lambda_1}u_0=\ldots=R_{\Lambda_{k-1}}u_0=0\rbrace.$$
	\end{enumerate}
	       \end{theorem}
\begin{proof}
We recall  elementary estimates for regular solutions and large $t$:
 \beq\label{lam4}
 \int_t^\infty \|u(\tau)\|^2\le C|u(t)|^2,\quad  \int_t^\infty \|Au(\tau)\|^2\le C\|u(t)\|^2.
 \eeq

Let $v(t)=u(t)/|u(t)|$. Note that $|v(t)|=1$ and $\|v(t)\|^2=\lambda(t)$.
One can derive a differential equation for $\lambda(t)$:
 \beq\label{lam1}
 \frac12\frac{d\lambda}{dt}+\nu |(A-\lambda)v|^2=-|u|\inprod{B(v,v)}{(A-\lambda)v}.
 \eeq
It follows that 
 \beqs
 \frac{d\lambda}{dt}+\nu |(A-\lambda)v|^2\le C\|u\| |Au|\lambda.
 \eeqs

 Neglecting the second term on the left-hand side, and using  Gronwall lemma  together with \eqref{lam4}, we obtain for sufficiently large $t'>t>0$ that
 \beq\label{lam3}
 \lambda(t')\le \lambda(t)e^{C\int_t^\infty \|u(\tau)\| |Au(\tau)|d\tau }
\le \lambda(t)e^{C' |u(t)| \|u(t)\|}.
 \eeq

 Using the fact that $|u(t)|$ and $\|u(t)\|$ go to zero as $t\to\infty$, and letting $t'\to\infty$, then $t\to\infty$, we obtain
 \beqs
0<\limsup_{t'\to\infty}\lambda(t')\le \liminf_{t\to\infty}\lambda(t)<\infty. 
 \eeqs
 Thus 
 \beqs
 \lim_{t\to\infty}\lambda(t)=\Lambda \text{ exists and belongs to }(0,\infty).
 \eeqs
 
 By \eqref{lam1}, $(A-\lambda)v\in L^2(t,\infty)$. Then there exist $t_j\to\infty$ such that $(A-\lambda(t_j))v(t_j)\to0$ and $v(t_j)\to\bar v$ in $H$. Thus $Av(t_j)\to \Lambda \bar v$. Since $A$ is a closed operator, this yields $\bar v\in \DA$ and $A\bar v=\Lambda \bar v$. Note also that $|\bar v|=1$, hence, $\Lambda\in \sigma(A)$.

After this eigenvalue $\Lambda$ is established, ones can prove in \cite[Proposition 1 and Lemma 1]{FS1} that
\beq\label{udec}
\|u(t)\|\le Ce^{-\nu\Lambda t},
\eeq
\beq\label{iminusR}
\lim_{t\to\infty}\|(I-R_{\Lambda})e^{\nu\Lambda t}u(t)\|= 0.
\eeq
It remains to deal with $R_{\Lambda}e^{\nu\Lambda t}u(t)$. We have
\beqs
\ddt (R_{\Lambda}e^{\nu\Lambda t}u(t)) + e^{\nu\Lambda t} R_{\Lambda}B(u(t),u(t))=0.
\eeqs
Hence, for $s>t>0$:
\beqs
e^{\nu\Lambda s}R_{\Lambda}u(s)-e^{\nu\Lambda t}R_{\Lambda}u(t)=-\int_t^s e^{\nu\Lambda \tau} R_{\Lambda}B(u(\tau),u(\tau))d\tau.
\eeqs

By the exponential decay \eqref{udec} of $u(t)$, we see that the right-hand side goes to $0$ as $s,t\to\infty$. Therefore, by Cauchy's criterion,
\beqs
\lim_{t\to\infty} e^{\nu\Lambda t} R_{\Lambda}u(t)\text{ exists and belongs to }R_\Lambda H.
\eeqs
Together with \eqref{iminusR}, this yields (ii).

The sets $M_k$'s can be defined as level sets $M_k=\Phi_k^{-1}(0)$, where the functions $\Phi_k$'s are as follows.
For $k=1$,  $\Phi_1:\mathcal R\to R_{\Lambda_1}V$ is given by
\beqs
\Phi_1(v)=R_{\Lambda_1}v-\int_0^\infty e^{\Lambda_1 t} R_{\Lambda_1}B(S(t)v,S(t)v)dt.
\eeqs
For $k\ge 2$, the functions  $\Phi_{k}:M_{k-1}\to R_{\Lambda_k}V$ is defined by
\beqs
\Phi_k(v)=R_{\Lambda_k}v-\int_0^\infty e^{\Lambda_k t} R_{\Lambda_k}B(S(t)v,S(t)v)dt.
\eeqs

Then the analyticity of $M_k$'s results from the analyticity of the mapping $(t,v)\to S(t)v$ which is  due to  \cite{F74}. 
\end{proof}

\begin{remark}
The following remarks are in order.
\begin{enumerate}[label={(\alph*)}]
 \item (L. Tartar) In the case where $\Omega$ is only bounded in one direction (so that Poincar\' e inequality holds), one has also that the $\lim_{t\to \infty}\lambda (t)=\Lambda ((u_0)$ exists and belongs to the spectrum of $A$, which is not necessarily discrete.

\item   One can prove that the rate of decay given by  $\Lambda(u_0)$ gives also the decay rate of higher Sobolev norms, and also  the convergence of $u(t)e^{t\Lambda(u_0)}$ in all  $H^s$ for any $s>0$, see   \cite{FS3,Gui}. We refer to \cite{Ghi} for various extensions of the convergence of the Dirichlet quotients to other situations, in particular Navier-Stokes and MHD equations on compact Riemannian manifolds.

\item  We recall that for any $u_0\in H$,  the NSE possesses a weak solution $u$  which becomes regular for $t$ sufficiently large. In this case the Dirichlet quotient $\lambda (t)$ converges to an eigenvalue $\Lambda (u(\cdot))$ of $\sigma (A)$ that,  by lack of uniqueness, depends \textit{a priori} on the whole solution $u$.
 
\item  The manifolds $M_k$'s are apparently the only known nonlinear manifolds  invariant under the the Navier-Stokes flow.

\item The geometry of parts of the $M_k$'s that are far from the origin is unknown (see however below for a specific property in the periodic case).
\end{enumerate}

\end{remark}

The invariant manifolds $M_k$'s can also be characterized as in the next result.

\begin{theorem}[Corollary 2 \cite{FS1}]\label{Mkchar}
 The necessary and sufficient condition for $u_0\in M_{k-1}$, or equivalently, $\Lambda(u_0)\ge \Lambda_k$, with $k\ge 2$, is
 \beq\label{Mchar}
 \lim_{t\to\infty} e^{\nu \Lambda_j t}R_{\Lambda_j}S(t)u_0=0 \quad\forall j=1,2,\ldots,k-1.
 \eeq
\end{theorem}

Note in \eqref{Mchar} that it only requires the projection $R_{\Lambda_j}$ of $S(t)u_0$, not the whole $S(t)u_0$.

\medskip
One has further results on properties of the manifolds $M_k$'s in the periodic case.

\begin{theorem}[Remark 7 \cite{FS1}, Theorem 2 and Proposition 4 \cite{FS2}]\label{unblin}
In the periodic case, each $M_k$ is a smooth analytic, truly nonlinear manifold in $\mathcal R$, and contains a \emph{linear} submanifold $L_k$ of infinite dimension. Consequently, $M_k$ is unbounded in $V$.
\end{theorem}
\begin{proof}
For the nonlinearity, we argue by contradiction. Suppose $M_k$ is linear. Then it must coincide with its tangential linear manifold at $0$, which is $M_k^{\rm lin}$.
Together with the invariance of $M_k$ under the semigroup $S(t)$, it follows that
$$\sum_{j=1}^k R_{\Lambda_j} B(v,v)=0\text{ whenever  }v\in \DA \text{ such that }\sum_{j=1}^k R_{\Lambda_j} v=0.$$
 By construction of an explicit counter example (see \cite{FS2}), this fact is shown to be not true.

The construction of the invariant  linear submanifolds $L_k$'s is based on special motions of the Navier-Stokes equations in the periodic case that we describe now.


For $\mathbf k=(k_1,\ldots,k_n)\in \Z^n\setminus\{0\}$ such that $k_1+\ldots+k_n=0$, we  consider 
\begin{equation}\label{Uform}
u(x,t)=(\varphi(\mathbf  k\cdot x,t),\ldots,\varphi(\mathbf k\cdot x,t))\quad (n\text { times}). 
\end{equation}
where $\varphi(y,t)$ is a scalar function with $y,t\in \R$.
One can verify that $${\rm div}\;u=0\text{ and }(u\cdot \nabla) u=0.$$

Then any (spatial) $L$-periodic solution $\varphi$ of the linear heat equation
$$\frac{\partial \varphi}{\partial t}-\nu|\mathbf k|^2\frac{\partial^2 \varphi}{\partial y^2}=0$$
leads to a  solution of the  NSE of the form $u(x,t)$ in \eqref{Uform} and $p=$ constant. 
Clearly, such a solution satisfies 
$$\|u(t)\|_m\le C_m e^{-\nu|\veck|^2 t}\quad \forall m\ge 0.$$
Thus, $u(0)\in M_k$ if $|\veck|^2>\Lambda_k$.

Based on the above observation, we set 
\begin{align*}
 \mathcal U=\lbrace u\in V:\; & u(x)=(\varphi(\mathbf k\cdot x),\ldots,\varphi(\mathbf k\cdot x)),\;  \veck=(k_1,\ldots,k_n)\in \Z^n,\\ &k_1+\ldots+k_n=0,\; \varphi(y) \text{ is $L$-periodic on $\R$}\rbrace.
\end{align*}

Then $\mathcal U\cap M_k^{\text{lin}}$  is an infinite-dimensional submanifold of $M_k.$
\end{proof}

\medskip
\begin{remark}\label{Maper}
The family of manifolds constructed in Theorem \ref{unblin} is extended to the following more general ones, which are also used to analyze the decay of the helicity, see \cite{FHN1,FHN2} and subsection \ref{helisub} below.

Consider the periodic case in $\R^3$. Let $\veca$ be a vector in $\R^3$ such that its orthogonal plane has nontrivial intersection with $\Z^3$, this means
\[\orthplane:=\{\veck\in\Z^3, \veck\cdot \veca = 0\}\ne \{0\}.\]
Define the linear manifold $\directmanifold$ in $V$ by
\beq\label{M-directional}
\directmanifold=\{u\in V:u=\sum_{\veck\in\orthplane} a_\veck e^{i\veck\cdot x},\ a_\veck \textrm{ is (complex) collinear to }\veca, \textrm{ for all } \veck\in\orthplane \}.
\eeq

For $u_0\in\directmanifold$, $u(t)=e^{-tA}u_0$ belongs to $\directmanifold$ for all $t\in[0,\infty)$ and solves the linearized NSE
\beqs
\begin{cases}
\bds\frac{du}{dt}\eds + Au =0,\quad t> 0,\\
u(0)=u_0\in V.
\end{cases}\eeqs
as well as the NSE \eqref{NS}. The fact that the nonlinear term $B(u(t),u(t))$ vanishes in this situation can be easily verified.
Therefore, $\directmanifold$ is an invariant linear manifold in $\solnset$. Clearly, the cardinality of $\orthplane$ is infinite, and hence $\directmanifold$ is infinite-dimensional.
\end{remark}

\medskip
\begin{remark} Regarding the structure of the set $\mathcal R$, we have the following remarks.
\begin{enumerate}[label={(\alph*)}]
 \item  Since $\mathcal R$ is an open set, Theorem \ref{unblin} implies that in the periodic case, it contains an {\it unbounded} open subset of $V$.  This fact was unknown before and a proof of this fact is still unknown when $\Omega $ is an arbitrary  smooth bounded open subset of $\R^3$.
Also, the construction of the linear manifolds $L_k$ is explicit. (See \cite{Z} for another construction of arbitrary large solutions in the three-dimensional periodic case.)  

\item See Bondarevsky (\cite{Bon}) for a construction of unbounded star-shaped subsets of  $\mathcal R$ when $\Omega$ is a bounded open set of  $\R^3$ (with  Dirichlet boundary conditions).

\item By totally different arguments, in the case of $\R^3,$ one can obtain the global existence of solutions starting from initial data with small low frequencies but allowing large high frequencies (oscillations). See Cannone, Meyer and Planchon (\cite{CMP}),  and  Chemin and Gallagher (\cite{CG1, CG2}). 
\end{enumerate}

\end{remark}

\begin{remark}
The Dirichlet  quotients (interpreted as  the ratio of the { \it enstrophy} over the  {\it energy}) have been used to study geophysical flows, in particular to give a precise mathematical sense (and justify) the physicists' {\it selective decay principle}:
\begin{quote}
After a long time, solutions of the quasi-geostrophic equations and/or the two-dimensional incompressible Navier-Stokes equations approach those states which minimize the enstrophy for a given energy.
\end {quote}
For more details we refer to \cite {MW1, MSW, MW2, Zhan, Zhan2}.
\end{remark}

\begin{remark}\label{backrmk}
In a totally different context, the Dirichlet quotients have been used in \cite{CFKM} to study the backward behavior of solutions to the  periodic Navier-Stokes equations with (non-potential) time-independent body forces. More precisely  it is proven there  that the set of initial data for which the solution exists for all negative times and has exponential growth is rather rich, actually it is dense in the phase space of the NSE, answering positively a question of Bardos-Tartar \cite{BT}.
\end{remark}

Coming back to the study of NSE with potential forces, it has been proven in \cite{FS5}, by extending a result of Hartman (\cite{Ha} chap. IX, th. 6.2) for ODE's, that the NSE  have invariant manifolds with "slow" decay. More precisely,

\begin{theorem}[\cite {FS5}]\label{slow}
For any $k=1,2,\ldots$ there exist an open neighborhood $U_k$ of $0$ in $\mathcal R$ and a submanifold $F_k$ without boundary of $U_k$ such that
\begin{enumerate}[label={\rm (\roman*)}]
\item $F_k$ is $C^1$ and analytic outside the origin.
\item {\rm dim}\,$F_k=m_1+m_2+\ldots+m_k = N_k.$
\item $F_k$ is invariant, {\it i.e.} $S(t)F_k\subset F_k,\; \forall t\geq 0.$
\item The tangent space to $F_k$ at the origin is $(R_1+\ldots+R_k)H.$
\item $v\in F_k\setminus \lbrace 0\rbrace$ implies $\Lambda(v)\leq \Lambda_k.$
\end{enumerate}
\end{theorem}

\begin{corollary}
$M_k\cap F_{k+1}$ is an invariant submanifold of $U_k$, of dimension $m_k$ that satisfies 
\begin{enumerate}[label={\rm (\roman*)}]
\item $\Lambda (u)=\Lambda_k$ for all $u\in M_k\cap F_{k+1}\setminus \lbrace 0\rbrace.$
\item $M_{k+1}, F_{k-1}$ and $M_k\cap F_{k+1}$ are transverse at $0.$
\end{enumerate}
\end{corollary}

\begin{remark}
\begin{enumerate}[label={(\alph*)}]
 \item  While the manifolds $M_k$'s are unique, this is not the case of the $F_k$'s.
\item In the $2D$-periodic case, one can take $F_1=R_1 H.$ This results from the fact that then the function $t\mapsto \|S(t)v\|^2/|S(t)v|^2$ is decreasing for any nonzero $v\in V.$ It is a consequence of \eqref{lam1} and the orthogonal properties
\beqs
\inprod{B(u,u)}{u}=0\quad\text{and} \quad  \inprod{B(u,u)}{Au}=0,
\eeqs
which make the right-hand side of \eqref{lam1} vanish.

\item For some specific examples (such as the viscous Burgers equation or a nonlocal version of it), one can prove that the manifolds $F_k$'s are global. 
\end{enumerate}
\end{remark}

\begin{remark}
We refer to \cite{CFL} for construction of invariant manifolds in a rather general setting, and, in particular, to its Appendix B5 for illuminating comments on slow manifolds.
\end{remark}

\begin{remark}
The nonlinear spectral manifolds have been used in \cite{MaW} to study asymptotic stability issues for the periodic two-dimensional Navier-Stokes equations.
\end{remark}

 The results in Theorem \ref{Dirichlet} suggest that one can go further and look for an asymptotic expansion of the solution. This will eventually lead to the normal form.

We first introduce a technical notion on the spectrum of $A$. More generally,

 
 
 \begin{definition}
 Let $A$ be a, possibly unbounded, linear operator in a space $X$ with spectrum $\sigma(A)$.
 \begin{enumerate}[label={\rm (\roman*)}]
  \item A resonance in  $\sigma(A)$ is a relation of the type
 \begin{equation}\label{reseq}
a_1\Lambda_1+a_2\Lambda_2+\ldots+a_k\Lambda_k=\Lambda,   
 \end{equation}
 for some  $\Lambda,\Lambda_1,\Lambda_2,\ldots, \Lambda_k\in\sigma(A)$,
 and some positive integers $a_1,a_2,\ldots,a_k$ with $a_1+a_2+\ldots+a_k\ge 2$.

 \item  If $\Lambda\in \sigma(A)$ satisfies \eqref{reseq}, then we say $\Lambda$ is resonant.
 
 \item If $\sigma(A)$ has a resonance then we say it is resonant, otherwise nonresonant. 

 \item In case $A$ is the Stokes operator with the spectrum described in \eqref{Stkspec}, and $\Lambda=\Lambda_{k+1}$ for some $k\ge 1$, then \eqref{reseq} is equivalent to
\beqs
a_1\Lambda_1+a_2\Lambda_2+\ldots+a_k\Lambda_k=\Lambda_{k+1}, \quad\text{for some } a_1,a_2,\ldots,a_k\in \N\cup\{0\}.
\eeqs
 \end{enumerate}
 \end{definition}

 \medskip
We note that the periodic boundary conditions  on  $\lbrack 0,L\rbrack^n$ always lead to resonances because $\Lambda_2=2\Lambda_1$.
 On the other hand, for periodic boundary conditions on  general cubes $\lbrack 0,L_1\rbrack \times \lbrack 0,L_2\rbrack \times \lbrack 0,L_3\rbrack$, the  spectrum of the Stokes operator $A$ is non-resonant for a dense set of periods $(L_1,L_2,L_3)\in (0,\infty)^3$.
 
 It has been recently proven (\cite{CKL}) that in the case of Dirichlet boundary conditions, the spectrum of $A$ is non-resonant generically with respect to the domain. More precisely let $\D^3_l$ be the set of bounded domains in $\R^3$ with $C^l$ boundary equipped with a suitable topology. For any $\Omega \in \D^3_l$ we denote by $\D_l^3(\Omega)$ the Banach manifold obtained as the set 
 of images $(Id+u)(\Omega)$  by $u\in W^{l+1, \infty}(\Omega, \R^3)$ which are diffeomorphic to $\Omega.$
 
 The main result in \cite{CKL} is that generically with respect to $\Omega\in \D^3_5$ the spectrum of $A$ is non-resonant, in the sense that the set of domains in $\D_5^3(\Omega)$ for which  the non-resonance property holds contains an intersection of open and dense subsets of $\D_5^3(\Omega).$ This result is established as a consequence of the fact that generically with respect to $\Omega \in\D_4^3,$ the eigenvalues of $A$ are simple.

 \section{The  asymptotic expansion}
 \label{secexpansion}
 
In this section we obtain asymptotic expansions, as time tends to infinity, for regular solutions of NSE. These expansions are of the following type.

 \begin{definition}\label{expanddef}
Let  $X$ be a real vector space.

{\rm (a)} An $X$-valued polynomial is a function $t\in \R\mapsto \sum_{n=1}^d a_n t^n$,
for some $d\ge 0$, and $a_n$'s belonging to $X$.

{\rm (b)} When $(X,\|\cdot\|)$ is a normed space, a function $g(t)$ from $(0,\infty)$ to $X$ is said to have the asymptotic expansion
	\beqs
g(t) \sim \sum_{n=1}^\infty g_n(t)e^{-\alpha_n t} \text{ in } X,
\eeqs
where $(\alpha_n)_{n=1}^\infty$ is a strictly increasing sequence of positive numbers, $g_n(t)$'s are $X$-valued polynomials, if for all $N\geq1$, there exists $\varepsilon_N>0$ such that
\beqs
\Big\|g(t)- \sum_{n=1}^N g_n(t)e^{-nt}\Big\|=\mathcal O(e^{-(N+\varepsilon_N)t})\ \text{as }t\to\infty.
\eeqs
\end{definition}

Throughout, $A$ is the Stokes operator.

 \subsection{The non-resonant case}
     Assume $\sigma(A)$ is non-resonant.

 \begin{theorem}[Theorem 2 \cite{FS3}]\label{resexp}
 Let $u$ be a regular solution. For each $N\in \N$, one has the expansion in $H$: 
 $$u(t)=W_{\mu_1}e^{-\nu \mu_1 t}+W_{\mu_2}e^{-\nu \mu_2 t}+\ldots+W_{\mu_N}e^{-\nu \mu_N t}+v_N(t),\quad \forall t>0,$$
 where $W_{\mu_j}=W_{\mu_j}(u_0)\in \mathcal E^\infty(\Omega)\cap V$ for $j=1,\ldots,N$, and
\beq \label{vNsmooth}
v_N\in C([ 0,\infty;V)\cap L^2_{\rm loc}(0,\infty;D(A))\cap C^\infty([ t_0,\infty);\mathcal E^\infty(\Omega)\cap V),\quad \forall t_0>0.
\eeq
 Moreover,
 \begin{enumerate}[label={\rm (\roman*)}]
 \item $\|v_N(t)\|_m=O(e^{-\nu(\mu_N+\epsilon_N)t}),$ for some $\epsilon_N>0$, $m=0,1,2,\ldots$.
 \item $R_jW_{\Lambda_j}=W_{\Lambda_j}$ for $\Lambda_j\leq \mu_N.$
 \item For $ \mu_j=\alpha_1\Lambda_1+\ldots+\alpha_{j-1}\Lambda_{j-1}$, with $\alpha_1+\ldots+\alpha_{j-1}\geq 2,$
 $W_{\mu_j}$ is a polynomial of $W_{\Lambda_1},\ldots,W_{\Lambda_{j-1}}$ which is homogeneous of degree $\leq \alpha_1$ in $W_{\Lambda_1},$ of degree $\leq\alpha _2$ in $W_{\Lambda_2}$, $\ldots$, of degree $\leq \alpha_{j-1}$ in $W_{\Lambda_{j-1}}.$
More precisely, one has in this case
$$\nu(A-\mu_jI)W_{\mu_j}+\sum_{\mu_i+\mu_k=\mu_j}B(W_{\mu_i},W_{\mu_k})=0.$$
\end{enumerate}
 \end{theorem}

 
 \begin{remark}
It is clear that the case  $W_{\mu_1}=W_{\mu_2}=\ldots =W_{\mu_{j-1}}=0$ and $W_{\mu_j}\neq 0$ corresponds to $\mu_j=\Lambda(u_0).$
\end{remark}

The rather technical proof is by induction on N, see details in \cite{FS3}.

 \subsection{The resonant case}
    Assume $\sigma(A)$ is resonant.

 \begin{theorem}[Theorem 4 \cite{FS3}]\label{AsRe} 
 Let $u$ be a regular solution with initial data $u_0\in \mathcal R$. For any  $N\in \N$, one has the asymptotic expansion in  $H$:
 $$u(t)=W_{\mu_1}(t)e^{-\nu \mu_1 t}+W_{\mu_2}(t)e^{-\nu \mu_2 t}+\ldots+W_{\mu_N}(t)e^{-\nu \mu_N t}+v_N(t),\quad \forall t>0,$$
 where $W_{\mu_j}(t)=W_{\mu_j}(t;u_0)$ is a $ V\cap\; \mathcal E^\infty(\Omega)$-valued   polynomial   in $t$,  and 
 $v_N$ satisfies \eqref{vNsmooth}.
 Moreover,
 \begin{enumerate}[label={\rm (\roman*)}]
 \item $\|v_N(t)\|_m=O(e^{-\nu(\mu_N+\epsilon_N)t}),$ for some $\epsilon_N>0$, $m=0,1,2,\ldots$.
 \item $d_j^0={\rm deg}\;W_{\mu_j}\leq j-1,\quad j=1,\ldots,N.$
 \item If $\Lambda_j\leq \mu_N$ is a non-resonant eigenvalue, then $W_{\Lambda_j}$ is constant in  $t$ and $R_jW_{\Lambda_j}=W_{\Lambda_j}.$
 \item If $\mu_j\leq \mu_N$ is not a non-resonant eigenvalue, then $W_{\mu_j}(t)$ satisfies the  equation  
 $$\frac{dW_{\mu_j}(t)}{dt}+\nu(A-\mu_j)W_{\mu_j}(t)+\sum_{\mu_l+\mu_k=\mu_j}B(W_{\mu_l}(t),W_{\mu_k}(t))=0,\quad \forall t\in\R.$$
 

 \item If $\Lambda_j$ is  a resonant eigenvalue, one has
 $${\rm deg}\; W_{\Lambda_j}\leq \max_{\mu_l+\mu_k=\Lambda_j}(d_l^0+d_k^0)+1.$$
 Moreover, $R_kW_{\Lambda_j}(t)$, for $k\neq j$ and the   coefficients of order $\geq 1$ in $R_jW_{\Lambda_j}(t)$ are obtained from $R_1W_{\Lambda_1}(0),\ldots,R_{j-1}W_{\Lambda_{j-1}}(0),$ via successive integrations of explicit elementary functions.
 
 \item If $\mu_j\notin\sigma(A),$ $W_{\mu_j}(t)$  is obtained from  $R_1W_{\Lambda_1}(0),\ldots,R_{j-1}W_{\Lambda_{k_j}}(0),$ via successive  integrations of elementary explicit functions, where 
 $$\Lambda_{k_j}=\max\; \lbrace \Lambda\in\sigma(A); \Lambda<\mu_j \rbrace.$$
One has also ${\rm deg}\;W_{\mu_j}=d_j^0\leq \sup_{\mu_l+\mu_k=\mu_j}(d_l^0+d^0_k).$
  \end{enumerate}
 \end{theorem}
 \begin{proof}
  The proof is also technical and   by induction on N. We merely sketch  the main steps.
  
  $\bullet$ First step. We recall that
  \beqs
  \|u(t)\|=O(e^{-\nu\Lambda_1 t}).
  \eeqs
 Ones can prove the limit in $V$:
  \beqs
   \lim_{t\to \infty}e^{\nu\Lambda_1 t} R_{\Lambda_1}u(t)=\xi_1\in R_{\Lambda_1}H,
  \eeqs
and then establish
  \beqs
   e^{\nu\Lambda_1 t} \|(I-R_{\Lambda_1})u(t)\|_m=O(e^{-\delta t})\text{ for some }\delta>0,\quad\forall m\ge0 . 
  \eeqs
  
 $\bullet$  Induction step. Let $v_N(t)=u(t)-\sum_{j=1}^N W_{\mu_j}(t)e^{-\nu\mu_j t}$.
  Assume
  \beqs
  \|v_N(t)\|_m=O(e^{-\delta t})\text{ for some }\delta>0,\quad\forall m\ge0 .
  \eeqs

  Write the equation for $w_N=e^{\nu\mu_{N+1}t} v_N(t)$ as
  \beqs
  \frac{d w_N}{dt}+ (A-\mu_{N+1})w_N + \sum_{\mu_\ell+\mu_j=\mu_{N+1}}B(W_{\mu_\ell}(t),W_{\mu_j}(t))=h_N(t),
  \eeqs
  where
  \beqs
    \|h_N(t)\|_m=O(e^{-\delta t})\text{ for some }\delta>0,\quad\forall m\ge0 .
  \eeqs

  We apply the projector  $R_{\lambda_k}$ and obtain the equation for $w_{N,k}=R_{\Lambda_k}w_N$:
  \beq\label{wNk}
  \frac{d w_{N,k}}{dt}+ (\Lambda_k-\mu_{N+1})w_{N,k} + p_{N,k}(t)=R_{\Lambda_k}h_N(t),
  \eeq
where $p_{N,k}$ is a polynomial in $t$.

This equation is an ODE of the type:
  \beqs
  \frac{d w}{dt}+ \alpha w + p(t)=g(t)=O(e^{-\delta t})\text{ in }R_{\Lambda_k}H,
  \eeqs
  where $p(t)$ is a polynomial. When either $\alpha\ge 0$, or $\alpha<0$ with $\lim_{t\to\infty}(e^{\alpha t} w(t))=0$, there exists a polynomial solution $q(t)$ of
  \beqs
  \frac{d q}{dt}+ \alpha q + p(t)=0
  \eeqs
such that
\beqs
|w(t)-q(t)|=O(e^{-\delta' t})\quad\text{for some }\delta'\in(0,\delta).
\eeqs

Using this fact we approximate $w_{N,k}$ in \eqref{wNk} by a polynomial $q_{N+1,k}\in R_{\Lambda_k}H$. We then define
$W_{\mu_{N+1}}(t)=\sum_{k=1}^\infty q_{N+1,k}(t).$
The function  $W_{\mu_{N+1}}(t)$ is proved to be a polynomial and  satisfies
\beqs
\|w_N(t)-W_{\mu_{N+1}}(t)\|_m=O(e^{-\varepsilon t})\quad\text{for some }\varepsilon>0,\quad\forall m\ge0 .
\eeqs
This implies
\beqs
\|v_N(t)-W_{\mu_{N+1}}(t) e^{-\mu_{N+1}t}\|_m=O(e^{-(\nu \mu_{N+1}+\varepsilon) t}),\quad\forall m\ge1 ,
\eeqs
which proves the induction step.

$\bullet$ Note that, in dealing with the higher norms $\|\cdot\|_{m}$, the proof in \cite{FS3} estimates  $\|d^{(j)}u/dt^{j}\|_{m}$ for all $j\ge 1$.
\end{proof}

\begin{remark}
The first coefficient $W_{\mu_j}(t)$ which is not identically zero in the expansion corresponds to $\mu_j=\Lambda(u_0).$ In this case it is constant in $t$ and belongs to $R_{\Lambda(u_0)}H.$
  \end{remark}

  \textbf{Notation.} 
  Based on Theorems \ref{resexp} and \ref{AsRe}, and according to Definition \ref{expanddef}, we have 
  \beqs 
  u(t)\sim \sum_{j=1}^\infty W_{\mu_j}(t)e^{-\nu\mu_j t} \quad \text{in } H^m(\Omega)^n,\quad \forall m\in \N,
  \eeqs  
which we will simply write 
  \beq\label{expnot}
  u(t)\sim \sum_{j=1}^\infty W_{\mu_j}(t)e^{-\nu\mu_j t}. 
  \eeq

   \subsection{The asymptotic expansion in Gevrey spaces}
   Theorem \ref{AsRe} has been  recently improved in  \cite{HM1} where it is proved  in particular that the asymptotic expansion in the 3D-periodic case actually holds in all Gevrey classes.

Consider the periodic case \eqref{peridom} with $n=3$ and assume \eqref{scalength}. Recall that one has the properties \eqref{scalesig}.
We first describe the relevant Gevrey classes.

For $\alpha, \sigma \in \R,$ and $u=\sum_{{\bf k}\in\Z^3\setminus\{0\}} \hat {u}({\bf k})e^{i{\bf k}\cdot x},$ we define

$$A^\alpha u=\sum_{{\bf k}\in\Z^3\setminus\{0\}} |{\bf k}|^{2\alpha} \hat{u}({\bf k})e^{i{\bf k}\cdot x},
$$
and the Gevrey class
$$G_{\alpha,\sigma}=D(A^\alpha e^{\sigma A^{1/2}})=\lbrace u\in H;|u|_{\alpha, \sigma} \eqdef |A^\alpha e^{\sigma A^{1/2}}u|<\infty\rbrace,$$
so that the domain of $A^\alpha$ is $D(A^\alpha)=G_{\alpha,0}$. Also $D(A^0)=H,$ $D(A^{1/2})=V.$

   The next theorem improves Theorem \ref{AsRe} for the periodic case for any weak solution.

\begin{theorem}[Theorem 1.1 \cite{HM1}]\label{Gevrey}
The expansion in  Theorem \ref{AsRe} holds in any Gevrey class $G_{\alpha,\sigma}$ with $\alpha, \sigma >0.$ More precisely for any (Leray-Hopf) weak solution $u$ of the NSE, there exist polynomials $q_n(t)$'s in $t$ valued in $\mathcal V$ such that if $\alpha, \sigma >0$ and $N\geq 1$ then
\begin{equation}\label{AsGev}
|u(t)-\sum_{n=1}^Nq_n(t)e^{-nt}|_{\alpha,\sigma}=O(e^{-(N+\epsilon)t})\quad\text{as }t\to \infty,\quad \text{for any}\;\epsilon\in (0,1).
\end{equation}
\end{theorem}

By working with the Gevrey norms, the proof in \cite{HM1} can avoid the estimates of $\|d^{(j)}u/dt^{j}\|_{m}$ for all $j\ge 0$ and $m\ge 0$.

\medskip
Before moving to the normal form theory for the NSE, we review other applications of the Dirichlet quotients techniques and asymptotic expansions.

\subsection{Application: asymptotic behavior of the helicity}\label{helisub}
It turns out that the techniques developed to study the Dirichlet quotients can be used to obtain information on the asymptotic behavior of the helicity for Navier-Stokes equations with potential forces. This is the object of the papers \cite{FHN1, FHN2}. We will focus on the results of \cite{FHN1}  where the (3D, periodic) deterministic case is considered. (Interested readers can read  \cite{FHN2} which deals with the statistical case). 

 
 For a regular solution $u$ of the NSE, the helicity is a scalar quantity defined by
 $$\mathcal H(t)=\int_{\Omega} u(x,t)\cdot\vecvor(x,t) dx,\quad \text{where }\vecvor= \text{curl}\; u.$$

   In the inviscid case ($\nu =0$) the invariance of the helicity was noticed by Moreau \cite{Mor}. The first thorough study of the helicity and of its density for inviscid incompressible flows was carried out by Moffatt \cite{Mof}, who gave in particular a connection of the helicity to the topological invariants and dynamics of the vortex tubes as well as the first examples of physically relevant fluid flows with non-zero helicity. There is
a general agreement that helicity plays an important role in magneto-hydro dynamics, but not
in the dynamics of ‘neutral’ flows (that is, solutions of the Euler or the Navier-Stokes equations). However, theoretical, empirical and numerical evidence indicate that the helicity can
provide insights into the nature of the fluid flows, at least in the case when the viscosity is small and this motivates the present study.     
   
 In the periodic case with our choice  of $\Omega$ in \eqref{peridom} ($n=3$), we recall from \eqref{periLamb} that the first eigenvalue of the Stokes operator $A$ is $\Lambda_1=4\pi^2/L^2,$ and the other ones are among $n\Lambda_1$, with $n\in \N.$
    The previous results on the limit of the Dirichlet quotients together with Cauchy-Schwarz inequality imply that the helicity of a regular solution tends to  zero as $t\to\infty$, at least with a rate  $2\nu \Lambda_1 n_0,$ where  $n_0$ depends on the  initial data.
    However, due to possible changes of sign and cancellations in $u\cdot \vecvor$, this does not imply that it has the same decay rate   $2\nu \Lambda_1 n_0,$  as that of the energy.
In particular, it was not known whether the  helicity could change sign or vanish infinitely times as   $t\to\infty.$
    An answer to those questions is found in \cite{FHN1}.

     We further define related quantities
    \begin{align*}
     \mathcal E (t)&=\frac{1}{2}\int_\Omega |u(x,t)|^2dx\quad \text{(kinetic energy)},\\
\mathcal F (t)&=\int_\Omega |\vecvor (x,t)|^2 dx\quad \text{(rate of energy  dissipation/viscosity)}\\
\mathcal I(t)&=\int_\Omega \vecvor (x,t)\cdot(\nabla \times \vecvor(x,t))dx.
    \end{align*}

These entities satisfy the following (balance) equations:
     $$\frac{d}{dt} \mathcal E (t)+\nu \mathcal F(t)=0,$$
      $$\frac12\frac{d}{dt}\mathcal H(t)+\nu \mathcal I(t)=0.$$  
     
We now assume \eqref{scalength}, and recall  that \eqref{scalesig} holds true.
     
     \begin{theorem}[Theorem 3.1 \cite{FHN1}] \label{FHN1}
For any regular solution of the NSE,
   \begin{enumerate}[label={\rm (\roman*)}]
   \item Either  the helicity becomes non-zero and  decays as  $t^de^{-2h_0 t},$ where $d\geq 0$ and $h_0$ are integers depending on $u_0$ with 
   $$\lim_{t\to \infty} \frac{\mathcal I(t)}{\mathcal H(t)}= h_0,$$
   \item Or it is identically zero.
   \end{enumerate}
   \end{theorem}
\begin{proof}
   Using asymptotic expansion of $u$, we derive for the helicity that
   \beqs
   \mathcal H(t)\sim \sum_{j=1}^\infty \phi_j(t)e^{-jt},
   \eeqs
   where the $\phi_j$'s are polynomials in $t\in\R$.
   
   If one of $\phi_j$'s is not a zero polynomial, then we obtain case (i).
   Otherwise, $\mathcal H(t)$ decays to zero, as $t\to\infty$, faster than any exponential functions. This fact itself cannot yield conclusion for case (ii).
   More properties of the solution $u$ and helicity $\mathcal H$ are needed. For those, we complexify the NSE in time and denote the resulting solution and helicity by $u(\zeta)$ and $\mathcal H(\zeta)$, for the complex time $\zeta\in \mathbb C$.
   These functions are proved to be analytic and bounded in a domain $E$ which, see \cite[Propositions 8.3 and 8.4]{FHN1}, contains $(0,\infty)$ and an open set 
\beq\label{Dto} t_0+D,\eeq 
where $t_0$ is a certain positive time, and  
\beq\label{D} D=\{\tau+i\sigma\in \mathbb C:\tau>0,|\sigma|<\sqrt 2 \tau e^{\alpha\tau}\}\eeq
for some positive constant $\alpha$. 
Then, see \cite[Lemma B.2]{FHOZ1},  the transformation  
\[\varphi(\zeta)=\zeta-\frac{1}{\alpha}\log (1+\alpha\zeta)\]
conformally maps $D$ to a set containing the right-half plane $H_0$. Moreover, $$\varphi([0,\infty))=[0,\infty).$$

Define the function $\mathcal H_{t_0}(\zeta)=\mathcal H(t_0+\zeta)$.
We have $\mathcal H_{t_0}\circ \varphi^{-1}$ is analytic, bounded on $H_0$ and, see \cite[Lemma 8.5]{FHN1}, satisfies
    \beqs
   \limsup_{0<\eta\to\infty} e^{\beta \eta}|(\mathcal H_{t_0}\circ \varphi^{-1})(\eta)| \leq
\limsup_{0<\zeta\to\infty} e^{\beta \zeta}|\mathcal H_{t_0}(\zeta)|=0\quad \forall\beta>0.
\eeqs

Then applying Phragmen-Linderloff type estimates, see \cite[Proposition C.1]{FHN1}, we infer that  $(\mathcal H_{t_0} \circ \varphi^{-1})(\eta)= 0$ for all $\eta\in H_0$. This implies $\mathcal H(\zeta)=0$ on the  open, non-empty set $D^*=t_0+\varphi^{-1}(H_0)$, which, by analyticity, yields that $\mathcal H(t)= 0$ for all  $t\in(0,\infty)$.
\end{proof}
   
The next theorem shows that the case where the helicity is non-zero is generic.
     
    \begin{theorem}[Theorem 3.2 \cite{FHN1}] \label{FHN2}
   Let  $\mathcal R_1$ and $\mathcal R_0$ be the sets of initial data  $u_0\in \mathcal R$ corresponding  to cases (i) and (ii) in Theorem \ref{FHN1}, respectively.   
   Then  $\mathcal R_1$ is open and dense in  $\mathcal R$   while  $\mathcal R _0$ is closed and contains an infinite union of linear, closed infinite dimensional manifolds.
   \end{theorem}

The asymptotic decay of helicity is  precisely described in the next result.
   

   \begin{theorem}[Theorem 3.4 \cite{FHN1}]\label{FHN3}

    Let $u_0\in \mathcal R \setminus \lbrace 0 \rbrace$ and $n_0=\Lambda(u_0)$. Then
    $$\lim_{t\to \infty}\frac{\mathcal H(t)}{|u(t)|^2}=\alpha _0,  \quad \text{where} \; \alpha_0 \in \lbrack -n_0, n_0\rbrack.$$
    
    Moreover, for any  $n\in \sigma (A)$ and $\alpha \in \lbrack -\sqrt n,\sqrt n \rbrack,$ there exists $u_0\in \mathcal R$ such that the corresponding solution $u$ satisfies
    $$\Lambda(u_0)=n \text{ and }\lim_{t\to \infty} \frac{\mathcal H(t)}{|u(t)|^2}=\alpha.$$
   \end{theorem}
   
  Note that one has the norm relation $\|u(t)\|=| \omega(t)|$, hence,
      $$\Lambda(u_0)=\lim_{t\to \infty}\lambda(t)=\lim_{t\to \infty}\frac{\|u(t)\|^2}{|u(t)|^2}=\lim_{t\to \infty}\frac{|\omega(t)|^2}{|u(t)|^2}.$$

If $\alpha_0\neq 0$ in the previous theorem, one is in case  (i) of Theorem \ref{FHN1} (helicity decays) with $d=0$ and $h_0=n_0\in \sigma( A).$ The situation where  $\alpha_0=0$ is considered in the next theorem.
   
   \begin{theorem}[Theorem 3.5 \cite{FHN1}]\label{FHN4}
   For any $n\in \sigma (A)$ and $M>0,$ there  exists an  initial data $u_0\in \mathcal R$ such that one is in  case (i) of Theorem \ref{FHN1} with $n_0=n$ and $h_0\geq n_0+M,$ and such that
   $$\frac{\mathcal H(t)}{|u(t)|^2}=O(e^{-2Mt}) \quad \text{when}\; t\to \infty.$$
   
   Moreover, there exist   solutions whose  helicity satisfies the  condition    
   $$\lim_{t\to \infty} \mathcal H(t)t^{-d}e^{2h_0t} \quad \text{exists and is not zero},$$
   where $d > 0$ or $h_0$ is not an eigenvalue of $A$.
   \end{theorem}

Comments on the proofs of Theorems \ref{FHN2}--\ref{FHN4}.

   \begin{itemize}
   \item The properties of the set $\mathcal R_0$ of initial data leading to an identically zero helicity result from a study of the spectrum of the  curl operator and of a global stability result of NSE in 3D (\cite{PRST}).

\item Examples of linear submanifolds of $\mathcal R_0$ are, among others, the family $\directmanifold$ presented in Remark \ref{Maper}.
   \end{itemize}

  
     \section{The Poincar\'e-Dulac theory of normal forms}
   \label{secODE}
   
   This section briefly reviews the Poincar\'e-Dulac theory of normal forms.  This is, of course,  a classical topic in dynamical systems, initiated by Poincar\' e and, later, Dulac (see \cite{Dul, Poin}) to analyze the dynamics of a nonlinear system of ODEs in the neighborhood of a singular point. We refer  to Arnold's book \cite{Ar} for a modern treatment.

The next theorem is extracted from Poincar\' e's thesis (1879).
   
   \begin{theorem} [Poincar\' e's thesis 1879 \cite{HPThese}]\label{Pthm}
  
   If the eigenvalues of the matrix $A$ are nonresonant, the equation 
   \begin{equation}\label{PD1}
\frac{dx}{dt}+Ax+\sum_{d=2}^\infty \phi^{\lbrack d\rbrack}(x)=0,    
   \end{equation}
   where  each $\phi^{\lbrack d\rbrack}$ is a homogeneous polynomial of degree $d$ in $\R^n,$   
   reduces to the linear equation
  \begin{equation*}
   \frac{dy}{dt}+Ay=0
  \end{equation*} 
  by a formal change of variable   
   \begin{equation}\label{PDchange}
    x=y+\sum_{d=2}^{\infty}\psi^{\lbrack d\rbrack}(y),
   \end{equation}
    where  each $\psi^{\lbrack d\rbrack}$ is a homogeneous polynomial of degree $d$.   
   \end{theorem}
  
   Retrospectively the following extract of Bonnet and Darboux report on Poincar\' e's thesis (see \cite{Gis} page 331) is somewhat amazing:

  \begin{quote}
    ... Quelques lemmes de l'introduction ont aussi paru dignes d'int\' er\^{e}t. Le reste de la th\`ese est un peu confus et prouve que l'auteur n'a pu encore parvenir \`a exprimer ses id\' ees d'une mani\`ere claire et simple. N\' eanmoins la Facult\' e tenant compte de la grande difficult\' e du sujet et du talent qu'a montr\' e M. Poincar\' e lui a conf\' er\' e avec trois boules blanches le grade de docteur.
   \end{quote}
   
     \begin{quote}
       \textit{(Translation)} ... the remainder of the thesis is a little confused and shows that the author was still unable to express his ideas in a clear and simple manner. Nevertheless, considering the great difficulty of the subject and the talent demonstrated, the faculty recommends that M. Poincar\' e be granted the degree of Doctor with all privileges. 
\end{quote}


\medskip
The resonant case was treated in Dulac's thesis (1912).

    \begin{definition}\label{classical}
Suppose an $n\times n$ matrix $A$ has eigenvalues $\lambda_1,\ldots,\lambda_n$ and corresponding eigenvectors $\xi_1,\ldots,\xi_n$. For each $x\in\R^n$, let $x_i$ be its coordinate with respect to $\xi_i$.
A monomial $x_1^{\alpha_1} x_2^{\alpha_2}\ldots x_n^{\alpha_n} \xi_k$ of degree two or higher is called \highlight{resonant} if
\beqs
\lambda_k = \alpha_1 \lambda_1+\alpha_2 \lambda_2+\ldots+\alpha_n \lambda_n. 
\eeqs
\end{definition}

   
   \begin{theorem}    [Dulac \cite{Dul}]\label{Dthm}
   The equation \eqref{PD1} reduces, by a formal change of variable \eqref{PDchange}, to the canonical form
   $$\frac{dy}{dt}+Ay+\sum_{d=2}^\infty \Theta^{\lbrack d\rbrack}(y)=0,$$
   where all monomials  in each $\Theta^{\lbrack d\rbrack}$ are resonant.   
    \end{theorem}

\section{A normalization map and a normal form for NSE}
\label{secnormal}

One can use the   {\it generating part} of the asymptotic expansion to construct a  normalization of the  NSE.

We first define the   Fr\' echet  space 
\beq\label{SA} \mathcal S_A=R_1H\oplus R_2H\oplus \ldots\eeq 
endowed with the topology of   convergence of components. 
The Stokes operator $A$ extends trivially to $\mathcal S_A$.

    \subsection{The non-resonant case}\label{subnonres}
    Assume $\sigma(A)$ is non-resonant.
    
      \begin{theorem}[Theorem 3, Corollaries 1 and 2 \cite{FS3}]\label{norm1}
    Define the mapping $W:\mathcal R\to \mathcal S_A$  by 
        $$W(u_0)=W_{\Lambda_1}(u_0)\oplus W_{\Lambda_2}(u_0)\oplus \ldots \text{ for } u_0\in\solnset.$$
     Then:   
    \begin{enumerate}[label={\rm (\roman*)}]
    \item $W$    is analytic and one-to-one.
    \item $W$ linearizes the  NSE in the sense
    $$W(u(t))=e^{-\nu At}W(u_0),\quad \forall u_0\in\mathcal R, \quad \forall t\geq 0.$$
    
    \item $u_0\in M_k$ if and only if the  first $k$ components of $W(u_0)$ vanish.
    
    \end{enumerate}
     \end{theorem}
     
In this non-resonant case,
$v(t)=W(u(t))$ is, thanks to (ii), a  solution of the  {\it linear}  Navier-Stokes equations in the large space  $\mathcal S_A$, i.e.,
\beq\label{Steq}\frac{dv}{dt}+Av=0.\eeq

Thus, the mapping $W$ transforms the (nonlinear) NSE \eqref{NS} to the linear one \eqref{Steq}, which is the case of Theorem \ref{Pthm}. Therefore, we call $W$ a normalization map, even though it is not a formal series.

   \subsection{The resonant case}\label{subres}
       Assume $\sigma(A)$ is resonant.
   In view of the structure of the  asymptotic expansion in this case, considering the following mapping $W$ is natural.
   
   \begin{theorem}[Theorem 5 \cite{FS3}]\label{norm2}
   The mapping  $W: \mathcal R \to \mathcal S_A$ given by   
   $$W(u_0)=R_1W_{\Lambda_1}(0;u_0)\oplus R_2W_{\Lambda_2}(0;u_0)\oplus R_3W_{\Lambda_3}(0;u_0)\oplus\ldots \text{ for } u_0\in\solnset,$$
   is analytic and one-to-one.
 \end{theorem}
   
We will see that this mapping $W$ also plays the role of a normalization map. First, we find important polynomials that are essential in transforming the NSE into a normal form.

\begin{lemma}[Lemma 7, \cite{FS3}]\label{mPjlem}
 For every $j=1,2,3,\ldots$, there exists a multilinear function $\mP_j$, defined on $R_1H\oplus\ldots\oplus R_{k_j}H$, depending  on $\sigma (A)$, $B$, $\nu$, such that
\beqs
W_{\mu_j}(u_0)=\mP_j(W_1(u_0),W_2(u_0),\ldots,W_{k_j}(u_0)).
\eeqs
\end{lemma}

\medskip
According to Lemma \ref{mPjlem},  each function $\mP_j$, for $j=1,2,3,\ldots,$ is a {\it polynomial}  with values in   $E^{\infty}\cap V$ (that can be explicitly constructed by induction), defined on $R_1H\oplus R_2H\oplus \ldots\oplus R_{k_j}H$.

\medskip
These polynomials have the following supplementary properties \cite[Lemma 7]{FS3}:   
   \begin{itemize}
\item  If $\mathcal M(x_1,x_2,..,x_{k_j})$ is a monomial in $\mP_j(x_1,\ldots,x_{k_j})$ of degree $m_1$, $\ldots$, $m_{k_j}$ in $x_1$,$\ldots$, $x_{k_j}$ respectively, then
   \beq \label{res1}
   m_1\Lambda_1+\ldots+m_{k_j}\Lambda_{k_j}=\mu_j.
   \eeq 
 \item   Furthermore, if $\mu_j$ is an eigenvalue $\Lambda_k,$  then 
   $$\mP_j(x_1,\ldots,x_{k_j})=x_k+\text{ higher order terms in }x_1,\ldots, x_{k-1}.$$
   \end{itemize}

   \medskip
With these polynomials $\mP_j$'s, we are ready to rewrite the NSE under the transformation $W(u(t))$.

   \begin{theorem}[Theorem 6 \cite{FS3}]
   Let $u_0\in \mathcal R$ and $u(t)=S(t)u_0$ be the corresponding global solution of the NSE.
    The  ($\mathcal S_A$-valued) function   $v(t)=W(u(t))$ satisfies the equation
   \beq \label{vnormal}\frac{dv(t)}{dt}+\nu Av(t)+\mathcal B (v(t))=0 \text{ in } \mathcal S_A,\eeq  
   where, for $v=v_1\oplus v_2\oplus\ldots\in\mathcal S_A$,   
   $$\mathcal B(v)=(\mathcal B_k(v))_{k=1}^\infty\in \mathcal S_A,\quad\text{with}\quad  \mathcal B_k(v) =\sum_{\mu_j+\mu_l=\Lambda_k} R_kB(\mP_l(v_1,\ldots,v_{k_l}),\mP_j(v_1,\ldots,v_{k_j})),$$
   where the polynomials $\mP_l$ and $\mP_j$ are defined as in Lemma \ref{mPjlem}. 
    \end{theorem}

   \medskip
   From the relation \eqref{res1}, one can prove also that each monomial in $\mB(v)$ is resonant, i.e.,
   if $\mathcal M(v_1,v_2,..,v_{k_j})$ is a (nonzero)  monomial in $\mB$ of degree $m_1,\ldots, m_{k_j}$ in $v_1,\ldots,v_{k_j}$ respectively,
   and $\mathcal M\in R_jH$, then
   \beq \label{res2}
   m_1\Lambda_1+\ldots+m_{k_j}\Lambda_{k_j}=\Lambda_j.
   \eeq 
   
   \medskip
   Thus, $W$ transforms NSE \eqref{NS} to \eqref{vnormal}, which satisfies the resonance condition as in Theorem \ref{Dthm}. Therefore, we, again, call $W$ a normalization map, and equation \eqref{vnormal} a normal form of NSE. 
   
   \medskip
   Although the normal form \eqref{vnormal} is nonlinear in $\mathcal S _A,$ it can be solved by successive integration of an  infinite set of non-homogeneous \emph{linear differential equations} in $R_kH,$ $k=1,2,\ldots,$  each one having an already known non-homogeneous part.
   
   \begin{remark}
  Minea (\cite{Min})  shows that  this type of normalization, when applied to ODEs, is a normalization in the sense of  Bruno (\cite{Bruno}).
   \end{remark}
     
   \subsection{Further results in the 3D periodic case}\label{further}
   The papers \cite{FHOZ1, FHOZ2} aim to answer the following questions.
   \begin{itemize}
   \item When does the asymptotic expansion actually converge?
   \item In what natural normed spaces is the normal form a well-behaved infinite-dimensional  system of ODEs?
   \item What is the range of the normalization map?
   
   \end{itemize}

Partial answers to those questions are established in the 3D periodic case, namely:

\begin{itemize}
   \item In paper (\cite{FHOZ1}): Construction of a suitable Banach space $\mathcal S_A^\star\subset \mathcal S_A$ on which the normal form is a well-posed system \textit{near the origin}. The norm $\|\bar u\|_\star$ of $\bar u=(u_n)^{\infty}_{n=1} \in \mathcal S_A^\star$ is
   \beq\label{starnorm}\|u\|_\star=\sum_{n=1}^{\infty} \rho_n\|\nabla u_n\|_{L^2(\Omega)},\eeq 
   where $(\rho_n)_{n=1}^{\infty}$ is a sequence of positive weights.
   
   \item In paper (\cite{FHOZ2}): choice of a suitable set of weights $\rho_n$ such that the normalization map $W:\mathcal R \to \mathcal S_A^\star$ is continuous and such that the normal form of the NSE is well-posed in \textit{the entire space} $\mathcal S_A^\star.$
   \end{itemize}

   \medskip
We consider thus the periodic case in $\R^3$ with the standard setting \eqref{peridom}, \eqref{scalength}, \eqref{scalesig} for $n=3$.

For $N\in \N$, we denote by $R_N$ the projection from $H$ onto the eigenspace of $A$ corresponding to $N$ in case $N\in\sigma(A)$, and set $R_N=0$ in case $N\not\in \sigma(A)$.

We note that the definition of $R_N$ is only different from that in subsection \ref{basic} by the change of the index.
This aims to unify calculations and make them more  efficient in lengthy proofs.

\medskip
The definitions of polynomials $\mP_j$'s and the normal form \eqref{vnormal} can be expressed more explicitly as follows.   
We recall that the asymptotic expansion \eqref{expnot} for a regular solution $u$ of the NSE  with initial data  $u_0\in \mathcal R $ is 
  $$u(t)\sim \sum_{j=1}^\infty q_j(t)e^{-jt},
$$
  where   $q_j(t)$'s are  polynomials in $t$ with values in   $\mathcal V$, and are unique polynomial  solutions of the following ODEs
  $$q'_j(t)+(A-j)q_j(t)+\beta_j(t)=0,\quad t\in \R,\quad R_jq_j(0)=W_j(u_0),$$
  with
  $$ \beta_1(t)=0,\quad \beta_j(t)=\sum_{k+l=j}B(q_k(t),q_l(t)),\quad \text{for}\; j>1.$$
   
   For  $\xi=(\xi_n)_{n=1}^{\infty} \in \mathcal S_A$ arbitrary, the polynomial solutions of the preceding system   with initial conditions   $R_jq_j(0)=\xi_j$ are explicitly given by the recursive formula:
   $$q_j(t,\xi)=\xi_j-\int_0^t R_j\beta_j(\tau)d \tau+\sum_{n\geq 0} (-1)^{n+1} (A-j)^{-n-1} \frac{d^n}{dt^n} (I-R_j)\beta_j, \quad j\in \N,$$
   where
   $$ (A-j)^{-n-1} u=\sum_{|{\bf k}|^2\not=j}\frac{a_{{\bf k}}}{(|{\bf k}|^2-j)^{n+1}} e^{i{\bf k}\cdot x},$$
   for $u=\sum_{|{\bf k}|^2\not=j}a_{{\bf k}} e^{i{\bf k}\cdot x}\in (I-R_j)H.$

   The   $\mathcal S_A$-valued function $\xi(t)=(\xi_j(t))^{\infty}_{j=1}=(W_j(u(t))_{j=1}^\infty=W(u(t))$ satisfies the system of ODEs  :
   
   \begin{equation}\label{NF}
  \left\lbrace
    \begin{aligned}
     &\frac{d\xi_1(t)}{dt}+A\xi_1(t)
	=0,\vspace{1mm}\\
     &\frac{d\xi_j(t)}{dt}+A\xi_j(t)+\sum_{k+l=j}R_jB(\mP_k(\xi(t)),\mP_l(\xi(t)))=0,\quad j>1.
     \end{aligned}\right.
     \end{equation}
     
     \bigskip
     Above, $\mathcal P_j(\xi)=q_j(0,\xi)$ for $\xi\in \mathcal S_A$ and $j\ge 1$. Then each  function   $\mathcal P_j(\xi)$,  for $\xi=(\xi_n)^\infty_{n=1},$ is a $\mathcal V$-valued polynomial in the variables $\xi_1,\xi_2,\ldots,\xi_j,$ each of which belongs to an finite dimensional space.
      For instance, 
      $$\mathcal P_1(\xi)=\xi_1,\quad  \mathcal P_2(\xi)=\xi_2-(A-2)^{-1}(I-R_2)B(\xi_1,\xi_1).$$

 We define $\mathcal B=(\mathcal B_j)^\infty_{j=1}$ where
   
   $$\mathcal B_1(\xi)=0,\quad \text{and } \mathcal B_j(\xi)=\sum_{k+l=j}R_jB(\mathcal P_k(\xi),\mathcal P_l(\xi))\quad\text{for } j>1.$$

We rewrite the system \eqref{NF} in a vectorial form in $\mathcal S_A$ as
  \begin{equation}\label{NFSA}
\frac{d\xi}{dt}+A\xi+\mathcal B(\xi)=0.   
  \end{equation}

Ones can verify that each monomial in $\mathcal B$ satisfies the resonance condition, see \eqref{res2}. 
Therefore, \eqref{NFSA} is a normal form  of  NSE in 
$\mathcal S_A$ under the transformation 
\beq \label{xiwu}
\xi=W(u).\eeq

The solution of the normal form \eqref{NFSA} with initial data $\xi^0=(\xi_n^0)_{n=1}^\infty\in S_A$ is precisely 
$$(R_nq_n(t,\xi^0)e^{-nt})_{n=1}^\infty.$$
We denote this solution by $S_{\rm normal}(t)\xi^0$. 

\medskip
Next, we investigate the convergence of the asymptotic expansion. We introduce and make use  of the following construction of regular solutions.
It is motivated by the asymptotic expansion itself.

We decompose the initial data $u^0$ in $V$ as 
\beq \label{inisplit}
u^0=\sum_{n=1}^{\infty} u^0_n. 
\eeq 
We find the solution $u(t)$ of the form
\beq \label{solsplit}
u(t)=\sum_{n=1}^{\infty} u_n(t),
\eeq 
where for each $n$, 
\beq\label{EQNN}
        \frac{du_n(t)}{dt}+A u_n(t) + B_n(t)=0,\quad t>0,
\eeq
with initial condition 
\beq\label{UN0COND}u_n(0)=u_n^0,\eeq
 where
\[ B_1(t)\equiv 0,\quad B_n(t)=\sum_{j+k=n} B(u_j(t),u_k(t)) \text{ for } n>1.\]

System \eqref{EQNN} will be called the extended NSE. 
We denote by $S_{\rm ext}(t)$ the semigroup generated by solutions of this system. 

\medskip
It turns out that such a construction \eqref{inisplit}--\eqref{UN0COND}, indeed, produces regular solutions of the form \eqref{solsplit} for NSE with the initial condition \eqref{inisplit}. The following existence theorem is a special case of Corollary 3.5 \cite{FHOZ1} with specific parameter $\rho_0>1=\rho$.

\begin{theorem}[Corollary 3.5 \cite{FHOZ1}]\label{familysolncor}
Let $(u_n^0)_{n=1}^\infty$ be a sequence in $V$ such that
\beqs \limsup_{n\to\infty}\norm{u_n^0}^{1/n} <1.\eeqs
Let $(u_n(t))_{n=1}^\infty$ be the solutions to \eqref{EQNN} and \eqref{UN0COND}, then  $u(t)=\sum_{n=1}^\infty  u_n(t)$ is the regular solution to the NSE with initial data $u^0=\sum_{n=1}^\infty u_n^0$, for $t\in[0,T)$, for some $T>0$.
\end{theorem}

We now make the connections between the solutions of the extended NSE with the asymptotic expansions of solutions of NSE.

\medskip
If a regular solution $u(t)$ has the expansion $\sum_{n=1}^\infty W_n(t,u^0) e^{-nt}$, then formally we wish for
\beqs
u^0=\sum_{n=1}^\infty W_n(0,u^0).
\eeqs

Therefore, we set $u^0_n=W_n(0,u^0)$ in the extended NSE. 
Then solutions $u_n(t)$ of the extended NSE are exactly $W_n(t,u^0)e^{-nt}$. Hence, conclusion in Theorem \ref{familysolncor} on $u_n(t)$, helps us make conclusion on $\sum_{n=1}^\infty W_n(t,u^0)e^{-nt}$.

First, we have a small initial data result.

\begin{theorem}[Proposition 5.9 \cite{FHOZ1}]\label{smallexpand}
There exists $\varepsilon_0>0$ such that if
\beq\label{cond1} \sum_{n=1}^\infty\norm{W_n(0,u^0)}< \varepsilon_0,
\eeq 
then $u(t,u^0)=\sum_{n=1}^\infty W_n(t,u^0)e^{-nt}$ is the regular solution to the NSE for all $t>0$.
\end{theorem}

Second, we have a large time result for large initial data.

\begin{theorem}[Theorem 5.10 \cite{FHOZ1}]\label{asympsoln}
Suppose 
\beq \label{cond2}\limsup_{n\to\infty}\norm{W_n(0,u^0)}^{1/n}< \infty.
\eeq 
Then there is $T>0$ such that \[v(t)=\sum_{n=1}^\infty  W_n(t,u^0)e^{-nt}\] is absolutely convergent in $V,$ uniformly in $t\in [T,\infty)$, $\sum_{n=1}^\infty W_n(t,u^0)e^{-nt}$ is the asymptotic expansion of $v(t)$, and 
\[u(t,u^0)=v(t) \textit{ for all } t\in [T,\infty).\]
\end{theorem}

Note in Theorem \ref{asympsoln} that it is not known whether the sum $\sum_{n=1}^\infty  W_n(t,u^0)e^{-nt}$ converges to a solution in short time.

\medskip 
Although the conclusions in Theorems \ref{smallexpand} and \ref{asympsoln} are satisfactory, it is not known whether the condition \eqref{cond1} or \eqref{cond2} holds true for a non-zero $u^0$. At the moment, we do not know whether
$\sum_{n=1}^\infty  W_n(0,u^0)$ and $\sum_{n=1}^\infty W_n(t,u^0)$ converge in $V$, i.e., with respect to the norm $\|\cdot\|$.
However, we hope to obtain some convergence in weaker norms.
Therefore, we study, in the following, the asymptotic expansions with a different approach, which uses suitable weighted normed spaces.

\medskip
Let $V^\infty=V\oplus V\oplus V\oplus\ldots$. Define 
$$W(t,\cdot):u\in\solnset\mapsto (W_n(t,u)e^{-nt})_{n=1}^\infty\in V^\infty,$$
$$Q(t,\cdot):\bar\xi\in S_A\mapsto (q_n(t,\bar\xi)e^{-nt})_{n=1}^\infty\in V^\infty.$$

We now proceed to the construction of the normed spaces.

\begin{definition}[Fast decaying weights]
Let $(\tilde\kappa_n)_{n=2}^\infty$ be a fixed sequence of real numbers in the interval $(0,1]$
satisfying
\beqs\label{kappaprime} \lim_{n\to\infty}(\tilde\kappa_n)^{1/2^n}=0.\eeqs
We define the sequence of positive weights $(\rho_n)_{n=1}^\infty$ by
\beqs\label{rhon-finaldef} \rho_1=1, \quad 
\rho_n=\tilde\kappa_n\gamma_n \rho_{n-1}^2,\ n>1,\eeqs
where $\gamma_n\in(0,1]$ are known and decrease to zero faster than $n^{-n}$.
\end{definition}

For $\bar u=(u_n)_{n=1}^\infty\in V^\infty$, let $\| \bar u \|_\star $ be defined by the formula \eqref{starnorm}.
Define the spaces
$$ V^{\star}=\{\,\bar u\in V^\infty:\norm{\bar u}_{\star}<\infty\,\}\quad\text{and}\quad
\SAstar=\mathcal S_A\cap V^\star.$$

Clearly $(V^\star,\| \cdot \|_\star) $ and $(\mathcal S_A^\star,\| \cdot \|_\star )$ are Banach spaces.

It turns out that the extended NSE is well-posed in the space $V^\star$.

\begin{theorem}If $\bar u^0\in V^\star$, then $S_{\rm ext}(t)\bar u^0\in V^\star$ for all $t>0$.
More precisely,
\beqs \norm{S_{\rm ext}(t)\bar u^0}_\star \leq Me^{-t},\ t>0,\eeqs
where
$M>0$ depends on $\rho_n$, $\kappa_n$ and $\norm{\bar u^0}_\star$.
\end{theorem}
\begin{theorem}\label{Sextconti} For each $t\in[0,\infty)$, $S_{\rm ext}(t)$ is continuous from $V^\star$ to $V^\star$.
More precisely, for any $\bar u^0\in V^\star$ and $\varepsilon>0$, there is $\delta>0$ such that
$$\norm{S_{\rm ext}(t)\bar v^0-S_{\rm ext}(t)\bar u^0}_\star < \varepsilon e^{-t},$$
for all $\bar v^0\in V^\star$ satisfying $\norm{\bar v^0-\bar u^0}_\star<\delta$ and for all $t\ge 0$.
\end{theorem}

As for the normal form and normalization map, one obtains the following well-posedness and continuity  results.

\begin{theorem}[Theorem 4.1 \cite{FHOZ2}]\label{normal-soln}
Let $\bar\xi=(\xi_n)_{n=1}^\infty\in \SAstar.$ Then $S_{\rm normal}(t)\bar\xi\in\SAstar$ for all $t\geq 0$.
Moreover,
\beqs \norm{S_{\rm normal}(t)\bar\xi}_\star\leq Me^{-t},\quad t>0,\eeqs
where $M$ is a positive number depending on $\norm{\bar\xi}_\star$ and 
the sequence $(\rho_n)_{n=1}^\infty$.
\end{theorem}

\begin{theorem}[Theorem 4.2 \cite{FHOZ2} and Theorem 7.4 \cite{FHOZ1}]\label{normalconti}
For each $t\geq 0$, the map 
$$\bar \xi\in\SAstar\to S_{\rm normal}(t)\bar\xi\in\SAstar\text{ is continuous.}$$
In particular, there exists $\varepsilon_0>0$ such that if $\bar\xi,\bar\chi\in \SAstar$ and $\norm{\bar\xi}_{\star},\norm{\bar \chi}_{\star}< \varepsilon_0$, then
\beqs
 \norm{S_{\rm normal}(t)\bar\xi-S_{\rm normal}(t)\bar\chi}_{\star}
          \leq 4e^{1/8}e^{-t}\norm{\bar \xi-\bar \chi}_{\star}\quad\forall t\ge 0.
\eeqs
\end{theorem}

According to Theorem \ref{normalconti}, the normal form \eqref{NFSA} is a well-posed system of ODEs in the infinite dimensional Banach space $\SAstar$.

In other words, the semigroup $S_{\rm normal}(t), t>0,$ generated by the solutions of the normal form \eqref{NFSA} leaves invariant the whole space $\SAstar$. Furthermore, we establish the continuity (but not necessarily Lipschitz continuity on the entire $\SAstar$) of each $S_{\rm normal}(t)$
as a map form $\SAstar$ to $\SAstar$, which means that the normal form is a well-posed system.

\begin{theorem}[Theorems 5.9  and 5.21 \cite{FHOZ2}]\label{main-normalmap}
The normalization map $W$ is a continuous function from $\solnset$ to $\SAstar$.
\end{theorem}



We summarize our results stated above in the commutative diagram (Figure \ref{main-diagram}) in which all mappings are continuous.

\begin{figure}[ht]
\begin{displaymath}\renewcommand{\labelstyle}{\textstyle}
\xymatrix@=1.75cm{
&\solnset \ar[ddl]_{W(0,\cdot)}\ar[d]^{W(\cdot)}\ar[rr]^{S(t)}&& 
\solnset \ar[d]_{W(\cdot)}\ar[ddr]^{W(0,\cdot)}\\
                   &{S_A^\star}\ar[dl]^{Q(0,\cdot)}      
\ar[rr]^{S_{\rm normal}(t)} && {S_A^\star}\ar[dr]_{Q(0,\cdot)}\\
           {V^\star} \ar[rrrr]^{S_{\rm ext}(t)} & & & & {V^\star}
}
\end{displaymath}
\caption{Commutative diagram for mappings and spaces (\cite{FHOZ2}).}
\label{main-diagram}
\end{figure}
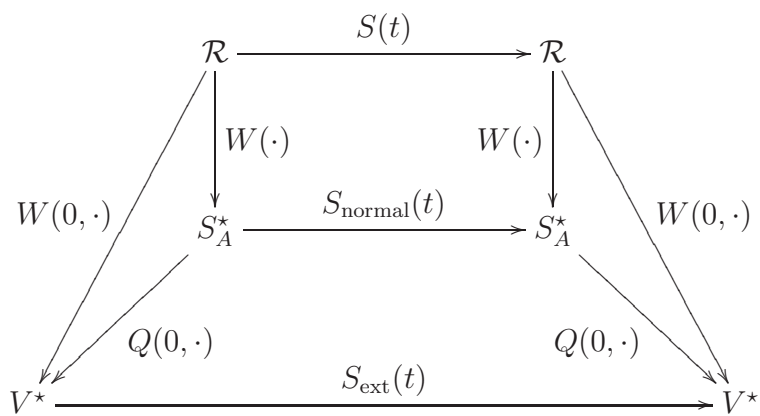
\noindent

The complete proofs of  Theorems \ref{normal-soln}, \ref{normalconti} and \ref{main-normalmap} are lengthy and technical and giving them  exceeds the scope of this survey. We refer the reader to the papers \cite{FHOZ1,FHOZ2} for details. They involve the complexification of NSE and extended NSE, for which the solutions are analytic in the complex time in a large domain of the form \eqref{Dto} and \eqref{D}. With appropriate transformation to transfer them to the half plane, and utilizing some Phragmen-Linderl\"off estimates, we can obtain recursive estimates for each step $W_n(u^0)$, $q_n(0,W(u^0))$, $q_n(\zeta,W(u^0))$,  and $u(\zeta)-\sum_{j=1}^n q_j(\zeta,W(u^0))e^{-j\zeta}$ for complex time $\zeta$. They, of course, depend on the weights $\rho_n$'s. Then the sum, say, $\sum_{n=1}^\infty \rho_n \|q_n(0,W(u^0))\|$ is convergent when $\rho_n$'s are chosen specifically and decay to zero extremely fast.


   \section{Navier and Stokes meet Poincar\' e and Dulac}
   \label{secPD}
    It was not totally clear that the normal form theory for the NSE derived in  section \ref{secnormal} could be related to the Poincar\' e-Dulac theory presented in section \ref{secODE}. It turns out to be the case, at least in the periodic case.

    \medskip
We consider thus the periodic case and use the same notation as in subsection \ref{further}.

In that subsection,  \eqref{NFSA} is a normal form  of  NSE  in a  suitable Banach space $\SAstar$. 
This space,  however, is too big to make a link with the {\it concrete} (formal series) approach of the   Poincar\' e-Dulac theory. Such link was finally  established  in \cite{FHS}.        
%
%
In short,
   \begin{itemize}
   \item The system \eqref{NFSA}, indeed, provides a  Poincar\' e-Dulac normal form of the  NSE, and is obtained by a (formal) explicit change of variables. The change of variables is a formal series expansion of the inverse of the normalization map $W$.
   
   \item Each homogeneous term in the formal series is well-defined in suitable  Sobolev spaces.
   \end{itemize}
   
\medskip   
We present below the precise results and provide main ideas and techniques in their proofs.

The following topological vector space  will be essential in our study 
\beq \label{Einfty}
E^\infty=C^\infty(\R^3,\R^3)\cap V\subset \mathcal S_A.
\eeq
It is endowed with the topology generated by the family of norms $|A^\alpha \cdot|$ for all $\alpha\ge 0$.

\medskip
First, we give an explicit definition of a normal form for the NSE, which is an analogue to classical ones by Poincar\'e and Dulac reviewed in section \ref{secODE}. We start with homogeneous polynomials  and resonant monomials in infinite-dimensional spaces.
      
\begin{definition}\label{polyxi}
Let $Q\in \mathcal H^{[d]}(E^\infty)$,  the space of homogeneous polynomials in $\xi\in E^\infty$ of order $d$. Then $Q(\xi)$ ($\xi\in E^\infty$ and $\xi_j=R_j\xi$, $j\in\N$),  is a monomial of degree $\alpha_{k_i}>0$ in $\xi_{k_i}$ where $i=1,2,\ldots,m$, $\alpha_{k_1}+\ldots+\alpha_{k_m}=d$ and $k_1<k_2<\ldots<k_m$ , if it can be represented as
\begin{equation}
\label{Qmonomial} 
  Q(\xi)=\tilde{Q}(\underbrace{\xi_{k_1},\ldots,\xi_{k_1}}_{\alpha_{k_1}},\underbrace{\xi_{k_2},\ldots,\xi_{k_2}}_{\alpha_{k_2}},\ldots,\underbrace{\xi_{k_m},\ldots,\xi_{k_m}}_{\alpha_{k_m}}), 
\end{equation}
where  
$\tilde Q(\xi^{(1)},\xi^{(2)},\ldots,\xi^{(d)})$ is a continuous $d$-linear map from $(E^\infty)^d$ to $E^\infty$.

The monomial $Q(\xi)$ defined by \eqref{Qmonomial}, with degree $d\ge 2$, is called \emph{resonant} if 
$$\sum_{i=1}^m \alpha_{k_i} k_i=j\text{ and }Q=R_jQ\ne 0.$$ 
\end{definition}

Although the  definition  of resonant monomials in Definition \ref{polyxi} is more abstract than that in Definition \ref{classical},  they are essentially the same, see details in \cite[Lemmas 4.4 and 4.6]{FHS}.

   \begin{definition}\label{NFdef}
    A differential equation in an infinite dimensional space $E$ 
\beq \label{gennormal} 
\frac{d\xi}{dt}+A\xi+\sum_{d=2}^\infty \Phi^{[d]}(\xi)=0
\eeq
is a Poincar\'e--Dulac normal form for the NSE if 
\begin{enumerate}[label={\rm (\roman*)}]
 \item Each $\Phi^{[d]}$ belongs to $\mathcal H^{[d]}(E)$, the space of homogeneous polynomials of order $d$,  and $\Phi^{[d]}(\xi)=\sum_{k=1}^\infty \Phi_k^{[d]}(\xi)$, where all $\Phi_k^{[d]}\in \mathcal H^{[d]}(E)$ are resonant monomials,

\item Equation \eqref{gennormal} is obtained from NSE by a formal change of variable 
\beq \label{genchange}
u=\sum_{d=1}^\infty \Psi^{[d]}(\xi),
\quad \text{where } \Psi^{[d]}\in \mathcal H^{[d]}(E).
\eeq 
\end{enumerate}
   \end{definition}

\bigskip   
To establish a normal form theory for the NSE, according to Definition \ref{NFdef}, we need to identify the framework $E$, the normal form  \eqref{gennormal}, and the formal change of variable \eqref{genchange}.

\medskip
\noindent\textbf{The framework.} We will use the  space  $E^\infty$ defined by \eqref{Einfty}. 

\medskip
\noindent
\textbf{The normal form.} 
The natural candidate for the normal form is \eqref{NFSA}. However, we must rewrite it in the power series form.

Let  $\mathcal  P_j^{\lbrack d\rbrack}(\xi)$ and $\mathcal B_j^{\lbrack d\rbrack}(\xi)$ denote the sum of all homogeneous monomials of degree $d$ of $\mathcal P_j(\xi)$ and $\mathcal B_j(\xi)$, respectively.
  Then the series  $\sum_j\mathcal P _j^{\lbrack d\rbrack}(\xi)$ and  $\sum_j\mathcal B _j^{\lbrack d\rbrack}(\xi)$ converge in   $E^\infty$ to continuous polynomials $\mathcal P ^{\lbrack d\rbrack}(\xi)$ and $\mathcal B ^{\lbrack d\rbrack}(\xi)$, respectively, see Theorem \ref{simple} below.

  The system \eqref{NFSA} is rewritten in the formal power series as
  \begin{equation}\label{NFPD}
  \frac{d\xi}{dt}+A\xi+\sum_{d=2}^\infty \mathcal B ^{\lbrack d\rbrack}(\xi)=0.
  \end{equation}

  Inheriting the spectral property \eqref{res2}, each polynomial $\mathcal B ^{\lbrack d\rbrack}(\xi)$  in \eqref{NFPD}
  can be verified to be resonant. Hence, system \eqref{NFPD} is a potential Poincar\'e-Dulac normal form, except that it is missing a power series change of variable.

  \medskip
\noindent\textbf{The formal change of variable.} We already know that NSE reduces to \eqref{NFPD} by the transformation $\xi=W(u)$. Therefore, $u$ should be $W^{-1}(\xi)$.
  Of course, $W^{-1}$ is not rigorously defined on $E^\infty$ and, additionally, not expressed in the power series form. 
 To resolve these, we heuristically argue that
  \begin{equation*}
u=\sum_{j=1}^\infty q_j(0,\xi)=\sum_{j=1}^\infty \sum_{d=1}^j q_{j}^{[d]}(0,\xi)
=\sum_{d=1}^\infty \highlight{\sum_{j=d}^\infty  q_{j}^{[d]}(0,\xi)}=\sum_{d=1}^\infty \highlight{\mP^{[d]}(\xi)}.
  \end{equation*}
Note that $\mP^{[1]}(\xi)=\xi$. Thus, the formal change of variable would be
 \begin{equation}\label{uxi}
 u=\xi+\sum _{d=2}^\infty \mathcal P^{\lbrack d \rbrack}(\xi).  
 \end{equation}

The change of variable \eqref{uxi} is considered as the formal inverse of the normalization map $W$.
  
  \medskip
  It turns out that, these arguments can be made rigorous and we obtain the following result.
  
  \begin{theorem}[Theorem 4.9 \cite{FHS}]\label{FHSthm}
  The system \eqref{NFPD}  is a  Poincar\' e-Dulac normal form in $E^\infty$ for the  NSE \eqref{NS}, and is obtained by the formal change of variable \eqref{uxi}.
   \end{theorem}

   The proof of Theorem \ref{FHSthm} relies on recursive formulas giving the homogeneous terms of the normal form. The main tool to estimate their Sobolev norms is the following  family of homogeneous gauges  $\lbrack \lbrack \xi \rbrack \rbrack_{d,n}$.

We introduce the set of general multi-indices
$GI=\bigcup_{n=1}^\infty GI(n)$ where for $n\ge 1$,
$$GI(n)= \big \{ \bar\alpha=(\alpha_k)_{k=1}^\infty, \, \alpha_k\in\{0,1,2,\ldots\},\, \alpha_k=0 \hbox{ for } k>n \hbox{ or } k\not\in \sigma(A) \big  \}.$$

For $\bar \alpha\in GI$, define 
\beq\label{indexnorms} |\bar\alpha|=\sum_{k=1}^\infty \alpha_k\quad \hbox{and}\quad \|\bar\alpha\|=\sum_{k=1}^\infty k\alpha_k.\eeq

\begin{definition}[Homogeneous gauges]
Let $\xi=(\xi_k)_{k=1}^\infty\in S_A$ and $\bar\alpha=(\alpha_k)_{k=1}^\infty\in GI$, define 
\beqs \sinorm{\xi}{\bar\alpha}=\prod_{\alpha_k>0} |\xi_k|^{\alpha_k}. 
\eeqs
For $n\ge d\ge 1$, define
\beqs   \dinorm{\xi}{d,n}
=\left (\sum_{|\bar\alpha|=d,\|\bar\alpha\|=n} \sinorm{\xi}{2\bar\alpha}\right)^{1/2}.\eeqs 
\end{definition}

One can easily compare $\dinorm{\xi}{d,n}$ with the usual  norm $|\xi|$ by
\beqs  \label{diPn} \dinorm{\xi}{d,n} \le  \Big ( \sum_{\bar\alpha\in GI(n),|\bar \alpha|=d} \sinorm{\xi}{2\bar\alpha} \Big )^{1/2}
\le |P_n \xi|^{d}.\eeqs 


Moreover, one has the following multiplicative and Poincar\'e inequalities.
\begin{lemma}[Lemma 2.1 \cite{FHS}]\label{dinorm-lem1}
Let $\xi\in S_A$, $n\ge d\ge 1$ and $n'\ge d'\ge 1$. Then
\beq  \label{prod-norm-ineq} \dinorm{\xi}{d,n} \,\cdot\, \dinorm{\xi}{d',n'}\le e^{d+d'} \dinorm{\xi}{d+d',n+n'}.\eeq
\end{lemma}

Note that the constant on the right-hand side of \eqref{prod-norm-ineq} is independent of $n,n'$.

\begin{lemma}[Lemma 2.2 \cite{FHS}]\label{dinorm-mainlem}
For any $\xi\in S_A$, any numbers $\alpha ,s\ge0$ and  $n\ge d\ge  1$, one has
\beqs  \label{imbed-ineq} \dinorm{A^\alpha\xi}{d,n}\le \left(\frac{d}{n}\right)^{s} \dinorm{A^{\alpha+s}\xi}{d,n}
\le \left(\frac{d}{n}\right)^{s}|P_n  A^{\alpha+s}\xi|^d.\eeqs  
\end{lemma}

The main advantage of the gauges $\dinorm{\cdot}{d,n}$ is that they efficiently track the norm contribution of each variable in estimates of homogeneous polynomials. It leads to simple and convenient bounds,  see \eqref{gaugest} below, for recursively defined, complicated $\mP_j^{[d]}(\xi)$ and $\mathcal B_j^{[d]}(\xi)$.

\medskip
\noindent\textbf{Convergence of homogeneous polynomials.}
For $d\ge 1$, let
\beq\label{Pddef} \mP^{[d]}(\xi) =\sum_{\highlight{j=d}}^{\highlight{\infty}} \mP_j^{[d]}(\xi)
= \sum_{j=d}^\infty q_{j}^{[d]}(0,\xi),
\eeq
and $d\ge 2$, let
\beq\label{Qbar}
\mB^{[d]}(\xi)=\sum_{\highlight{j=1}}^{\highlight{\infty}}\ \mB^{[d]}_j(\xi)= \sum_{j=1}^\infty\ \sum_{k+l=j}\ \sum_{m+n=d}\ R_j B\big( \mP_k^{[m]}(\xi),\mP_l^{[n]}(\xi)\big).
\eeq

\begin{theorem}[Theorems 3.4 and 3.5, Lemma 4.1  \cite{FHS}]\label{simple}
Let $\alpha\ge 1/2$.  Then $\mP^{[d]}(\xi)$, defined in \eqref{Pddef}  for $d\ge 1$, and $\mB^{[d]}(\xi)$, defined in \eqref{Qbar} for $d\ge 2$, are  continuous homogeneous polynomials  from $\mD(A^{\alpha+3d/2})$ to $\mD(A^\alpha)$, and satisfy  
\beq  \label{APd} |A^\alpha \mP^{[d]}(\xi)|\le  \sum_{j=d}^\infty  |A^\alpha \mP_{j}^{[d]}(\xi)| \le M(\alpha,d) |A^{\alpha+3d/2}\xi|^d,\eeq  
\beq\label{AQbar-ineq} |A^\alpha \mB^{[d]}(\xi)|\le \sum_{n=1}^\infty |A^\alpha \mB_n^{[d]}(\xi)| \le C(\alpha,d)|A^{\alpha+3d/2}\xi|^d,\eeq 
for some positive constants $M(\alpha,d)$ and $C(\alpha,d)$.

Consequently, the series $\mP^{[d]}(\xi)=\sum_{j=d}^\infty \mP^{[d]}_j(\xi)$ and $\mB^{[d]}(\xi)=\sum_{j=d}^\infty \mB^{[d]}_j(\xi)$,  converge in $E^\infty$, and are continuous homogeneous polynomials of degree $d$ from  $E^\infty$ to $E^\infty$.   
\end{theorem}

\begin{proof} 
By induction, ones first prove that
\beq\label{gaugest}   |A^\alpha \mP_j^{[d]}(\xi)|
\le c(\alpha,d) \dinorm{ A^{\alpha +\frac32(d-1)}\xi }{d,j}.\eeq 

Then using inequality \eqref{imbed-ineq},
\begin{align*}
&\sum_{j=1}^\infty  |A^\alpha \mP_{j}^{[d]}(\xi)|
\le \sum_{j=d}^\infty c(\alpha,d) \dinorm{A^{\alpha+\highlight{(3/2)(d-1)}}\xi}{d,j}\\
&\le \sum_{j=d}^\infty c(\alpha,d) \highlight{\Big(\frac d j\Big)^{3/2}} |A^{\alpha+(3/2)(d-1)+\highlight{3/2}}\xi|^d
= M(\alpha,d)  |A^{\alpha+3d/2} \xi|^d,
\end{align*}
which proves \eqref{APd}.
Inequality \eqref{AQbar-ineq} for $\mB^{[d]}(\xi)$ is proved similarly.
\end{proof}

\begin{proof}[Proof of Theorem \ref{FHSthm}]
 By virtue of Theorem \ref{simple}, the ODE \eqref{NFPD} and change of variable \eqref{uxi} are meaningful in $E^\infty$ now.
It still remains to prove that \eqref{NFPD}, indeed, comes from NSE and \eqref{uxi}. 
 One can derive from NSE \eqref{NS} the formal ODE for $\xi(t)$ under the transformation \eqref{uxi} as
\beq\label{NFxi} \frac{d \xi}{dt}+\sum_{d=1}^\infty Q^{[d]}(\xi)=0,\eeq
where $Q^{[1]}(\xi)=A\xi$ and, for $d\ge 2$,
\beqs
 Q^{[d]}(\xi)=\sum_{k+l=d} B( \mP^{[k]}(\xi),\mP^{[l]}(\xi))
 - \sum_{\substack{2\le k,l\le d-1\\k+l=d+1}} D\mP^{[k]}(\xi)(Q^{[l]}(\xi))
+H_A^{(d)} \mP^{[d]}(\xi),
\eeqs
with $$H_A^{(d)} \mP^{[d]}(\xi)=A\mP^{[d]}(\xi)-D\mP^{[d]}(\xi)A\xi$$ being the  Poincar\'e homology operator. 

It turns out that, see \cite[Proposition 4.7]{FHS}, 
$$Q^{[d]}(\xi)=\mB^{[d]}(\xi)\text{ for all }\xi\in E^\infty,\ d\ge 2.$$

Therefore, the transformed system \eqref{NFxi} is the same as \eqref{NFPD} whose resonance conditions are already met.
This implies that \eqref{NFPD} is a Poincar\'e-Dulac normal form of the NSE \eqref{NS} by the change of variable \eqref{uxi}.
\end{proof}

  
 \section{Final comments}
  \label{secFinal}
   
 \subsection{Other related results} 
\begin{enumerate}
\item The recent paper \cite{HM2}  obtains the asymptotic expansion of the same type for weak solutions of NSE in periodic domains with exponentially decaying (non potential) outer body forces satisfying an asymptotic expansion of the type
$$f(t)\sim \sum_{n=1}^{\infty} f_n(t)e^{-nt}$$ 
in appropriate (Gevrey type) functional spaces.
\item  When $\Omega=\R^n$, because of lack of the Poincar\' e inequality, the situation is drastically different and the decay rate is only algebraic. The techniques and their proofs are quite different than those used for the bounded domains. We refer to \cite{Schon, Mi-Scho, GW, Br1, Br2, Wie} and the references therein. 
In particular, Kukavica and Reis (\cite{KR}) obtain a precise space-time asymptotic of smooth solutions in a weighted space.

\item Section 7 of \cite{FS3} focuses on the viscous Burgers equation and the Minea system. In the case of the viscous Burgers equation, the normalizing mapping $W$ can be explicitly computed in terms of the Cole-Hopf transform
\item The asymptotic expansion is also established for dissipative wave equations by Shi in \cite{Shi}.
%
\end{enumerate}
  \subsection {Some open issues}
  
  We indicate below a few open questions related to the topics considered in the present paper. 
 \begin{enumerate}
 \item     The classical   Poincar\' e-Dulac theory for ODEs has a second part concerning the convergence of the formal series which give the change of variables, and the convergence of the series in the normal form (see \cite{Ar}). The extension of these convergence results to the NSE seems to be  an open problem.
 \item Complete our knowledge of the normalization map and of the normal form.
 \item What happens in $3D$ when the initial data  $u_0\in \mathcal R$ is near the boundary $\partial \mathcal R$ in case $ \mathcal R \neq V ?$
 \item It is very likely that the normal form theory studied here extends to the NSE posed on a compact Riemann manifold ({\it e.g.}~the Euclidean sphere in $\R^3$) (see \cite{Ghi} for results on the asymptotic decay). In particular, all specific results obtained in the periodic case should have a counterpart in this framework.
 \end{enumerate}

\section*{Acknowledgements}
JCS acknowledges  support by  the French Science Foundation ANR, under grant GEODISP.

LH  would like to thank the Departments of Mathematics at University of Tennessee (Knoxville), Indiana University (Bloomington), and Texas A\&M University (College Station) for their hospitality during his visits in the Fall of 2017, when he partly worked on the manuscript of this paper.



\bibliographystyle{amsplain}

\end{document}